\documentclass[a4paper,11pt,french,english]{amsart}
\usepackage[english]{babel}
\usepackage{lmodern}
\usepackage[T1]{fontenc}
\usepackage{amsmath,amsthm,amssymb,amsfonts}
\usepackage{amscd}
\usepackage[a4paper,left=3.8cm,right=3.8cm,top=4cm,bottom=4cm]{geometry}
\usepackage{enumerate}
\usepackage{enumitem}
\usepackage[linktocpage=true]{hyperref}
\usepackage{url}
\usepackage{csquotes}

\newcommand{\treed}{\mathcal{T}_d}

\newcommand{\arbmoins}{\mathcal{T}_g^-}
\newcommand{\arbplus}{\mathcal{T}_g^+}

\newcommand{\verti}{\mathsf{V}(\mathcal{T}_d)}

\newcommand{\edg}{\mathsf{E}(\mathcal{T}_d)}

\newcommand{\elle}{\mathsf{L}}

\newcommand{\Sy}{\mathrm{Sym}(\Omega)}

\newcommand{\Sx}{\mathrm{Sym}(X)}

\newcommand{\autd}{\mathrm{Aut}(\treed)}

\newcommand{\Ne}{\mathcal{N}}

\newcommand{\sg}{\sigma(g,v)}
\newcommand{\Sg}{S(g)}

\title{Groups acting on trees with almost prescribed local action}
\author{Adrien Le Boudec}
\address{Laboratoire de Mathématiques, Université Paris-Sud 11, 91405 Orsay, France}
\address{UCLouvain, IRMP, Chemin du Cyclotron 2, 1348 Louvain-la-Neuve, Belgium}
\email{adrien.leboudec@uclouvain.be}

\theoremstyle{plain}
\newtheorem{thm}{Theorem}[section]
\newtheorem{prop}[thm]{Proposition}
\newtheorem{cor}[thm]{Corollary}
\newtheorem{lem}[thm]{Lemma}

\newtheorem*{thm-intro}{Theorem}

\theoremstyle{definition}
\newtheorem{defi}[thm]{Definition}
\newtheorem{ex}[thm]{Example}
\newtheorem{rmq}[thm]{Remark}

\begin{document}

\maketitle

\begin{abstract}
We investigate a family of groups acting on a regular tree, defined by prescribing the local action almost everywhere. We study lattices in these groups and give examples of compactly generated simple groups of finite asymptotic dimension (actually one) not containing lattices. We also obtain examples of simple groups with simple lattices, and we prove the existence of (infinitely many) finitely generated simple groups of asymptotic dimension one. We also prove various properties of these groups, including the existence of a proper action on a CAT(0) cube complex. 
\end{abstract}




\section{Introduction} \label{sec-intro}

\subsection{Local action prescribed almost everywhere}


Let $\Omega$ be a set of cardinality $d \geq 3$ and $\treed$ a regular tree of degree $d$. 
Recall that the group $\autd$ of automorphisms of $\treed$, endowed with the permutation topology coming from the action on the set of vertices, is a totally disconnected locally compact group.

Given a permutation group $F \leq \Sy$, the Burger-Mozes' group $U(F)$ is the group of automorphisms of $\treed$ whose local action around every vertex is prescribed by $F$ \cite{BM-IHES}. 
The definition of the groups investigated in this paper can be seen as a relaxation of the definition of the groups $U(F)$, in the sense that the local action is prescribed almost everywhere only. More precisely, we let $G(F)$ be the subgroup of $\autd$ consisting of automorphisms whose local action is prescribed by $F$ for \textit{all but finitely many} vertices. The group $G(F)$ was first considered by the authors of \cite{BCGM} in the particular case when $F = \mathrm{Alt}(\Omega)$.

The group $U(F)$ is always closed in $\autd$, while $G(F)$ turns out to be dense in $\autd$ as soon as $F$ acts transitively on $\Omega$. Given a second permutation group $F' \leq \Sy$ containing $F$, we consider the group $G(F,F') = G(F) \cap U(F')$ consisting of automorphisms whose local action belongs to $F'$ for all vertices and to $F$ for all but finitely many of them. The group $G(F,F')$ always admits a natural group topology, which is defined by requiring that the inclusion $U(F) \hookrightarrow G(F,F')$ is continuous and open, where $U(F)$ is endowed with the induced topology from $\autd$. This topology turns $G(F,F')$ into a compactly generated totally disconnected locally compact group. The action of $G(F,F')$ on $\treed$ is continuous, but not proper in general. The motivation for considering these groups is precisely to eliminate the properness of the action on the tree, in order to build groups \textit{locally} isomorphic to $U(F)$ but with a significantly different structure. 


\subsection{Simplicity}

Recently the class $\mathcal{S}$ of compactly generated locally compact groups that are totally disconnected, topologically simple and non-discrete, has received much attention \cite{CaMo,CRW-1,CRW-2}. We refer the reader to the introduction of \cite{CRW-2} for the motivation and the most recent developments in the study of these groups. We prove in Section \ref{sec-simple} that the family of groups $G(F,F')$ contains many examples of groups that virtually belong to the class $\mathcal{S}$. Note that, although the class of groups of tree automorphisms was known to be a source of examples of groups in the class $\mathcal{S}$, see \cite[Subsection 1.1]{CRW-2}, all the examples mentioned therein are closed subgroups of the automorphism group of the tree, which is definitely not the case of $G(F,F')$.


\begin{thm} \label{thm-intro-ind2}
Let $F \leq F' \leq \Sy$ be permutation groups such that $F$ is transitive, and $F'$ is generated by the derived subgroups of its point stabilizers together with point stabilizers of $F$. Then $G(F,F')$ has a subgroup of index two that is simple.
\end{thm}

We point out that permutation groups satisfying these assumptions are abundant, see Section \ref{sec-simple} for examples.

\bigskip

For the sake of simplicity, the following result is not stated here in its more general form, and we refer to Theorem \ref{thm-simple-ind8} for a more comprehensive statement. We point out that Theorem \ref{thm-intro-ind8} had been previously obtained by Bader-Caprace-Gelander-Mozes in the case $F= \mathrm{Alt}(\Omega)$ (unpublished). We thank them for communicating their result. 

\begin{thm} \label{thm-intro-ind8}
Let $F \leq F' \leq \Sy$ be permutation groups such that:
\begin{enumerate} [label=(\alph*)]
	\item $F$ has index two in $F'$;
	\item $F$ is transitive and generated by its points stabilizers.
\end{enumerate}
Then $G(F,F')$ has a subgroup of index eight that is simple.
\end{thm}

\subsection{Small finitely generated simple groups}

The question whether the class of finitely generated simple groups contains examples of groups having a \enquote{small} geometry has recently received much attention. For example it has been proved that the derived subgroup of the topological full group associated to a minimal subshift, which is a finitely generated simple group \cite{Mat06}, is amenable \cite{JM} (see also \cite{Matte-Bon}). 

Here we prove that the family of groups $G(F,F')$ provides examples of finitely generated simple groups which are small from the point of view of asymptotic dimension. 

Recall that infinite finitely generated simple groups of finite asymptotic dimension are known to exist, as for instance it follows from the main result of \cite{asdim-free} that the finitely generated simple groups constructed in \cite{Camm} have finite asymptotic dimension. Finitely presented groups with these properties have moreover been constructed in \cite{BM-IHES-2} and \cite{Cap-Rem}. The following result provides the first examples of finitely generated simple groups of asymptotic dimension \textit{one}. 

\begin{thm} \label{thm-intro-fg-asdim}
Let $F \leq F' \leq \Sy$ be permutation groups such that:
\begin{enumerate} [label=(\alph*)]
	\item $F$ is simply transitive;
	\item $F'$ is generated by the derived subgroups of its point stabilizers.
\end{enumerate}
Then $G(F,F')$ has a subgroup of index two that is simple, finitely generated and of asymptotic dimension one.

Moreover there exist infinitely many isomorphism classes of groups generated by four elements and having these properties. 
\end{thm}

Theorem \ref{thm-intro-fg-asdim} is obtained by combining Theorem \ref{thm-intro-ind2} together with Corollary \ref{cor-gen-trans}, Corollary \ref{cor-asdimone} and Proposition \ref{prop-morph-trivial} (see Example \ref{ex-alt} for a family of examples generated by four elements).

\subsection{Lattices in simple groups}

The study of lattices in locally compact groups is of central interest, and experienced recent developments beyond the classical theory of Lie and algebraic groups. Of particular interest is the case of simple groups, and we refer to the introduction of \cite{BCGM} for the motivation. Up to now the only compactly generated simple group without lattices that is known is the group $\mathrm{AAut}(\treed)$ of almost automorphisms of a regular tree \cite{BCGM}. This group has a very rich geometry, and is also very large in the sense that it contains discrete $\mathbb{Z}^n$-subgroups for all $n$. In particular $\mathrm{AAut}(\treed)$ has infinite asymptotic dimension.

In Section \ref{sec-lattices} we study the existence of lattices in the groups $G(F,F')$. Recall that it follows from Bass-Kulkarni's theorem \cite{Bass-Kulk} that every \textit{closed} compactly generated unimodular $G \leq \mathrm{Aut}(\treed)$ admits cocompact lattices. By investigating certain locally elliptic groups which appear as union of infinitely iterated wreath products, we prove that some of the groups $G(F,F')$ contain no lattice (see Corollary \ref{cor-without-lattice}).

\begin{thm} \label{thm-simple-g(f,f')-no-lattice}
There exist permutation groups $F \leq F' \leq \Sy$ such that $G(F,F')$ does not contain lattices.
\end{thm}

There are natural permutation groups satisfying Theorem \ref{thm-simple-g(f,f')-no-lattice}, for example $F = \mathrm{PSL}(2,q)$ and $F' = \mathrm{PGL}(2,q)$ acting on the projective line $\mathbb{P}^1(\mathbb{F}_q)$, where $q=1 \! \! \mod 4$. We refer to the end of Section \ref{sec-lattices} for more examples. Theorem \ref{thm-simple-g(f,f')-no-lattice} shows that Bass-Kulkarni's theorem cannot be extended to a compactly generated unimodular group $G$ equipped with a continuous inclusion in $\mathrm{Aut}(\treed)$. Combined with other results of the paper, Theorem \ref{thm-simple-g(f,f')-no-lattice} also implies:

\begin{cor} \label{cor-simple-no-lattice}
Simple groups without lattices exist among compactly generated groups of asymptotic dimension one. 
\end{cor}

Nevertheless some of the groups $G(F,F')$ do have lattices. We actually prove the following result, which shows that among compactly generated \textit{simple} groups, having lattices is not invariant by passing to a closed cocompact subgroup (so in particular not invariant by quasi-isometry).

\begin{thm} \label{thm-simple-coc-no-lattice}
There exist totally disconnected locally compact compactly generated groups $H \leq G$ such that:
\begin{enumerate} [label=(\alph*)]
	\item $H$ is cocompact in $G$;
	\item $H$ and $G$ are abstractly simple;
	\item $G$ contains lattices but $H$ does not contain lattices.
\end{enumerate}
\end{thm}

Given a countable group $\Gamma$, the study of the envelopes of $\Gamma$, i.e.\ the groups that can contain $\Gamma$ as a lattice, is very natural since lattices generally reflect the properties of the ambient group. This problem is addressed in \cite{Dym} for certain solvable groups, and structure results of envelopes of a large class of countable groups have been announced in \cite{BFS}. Note that the groups $G(F,F')$ that are finitely generated have infinite amenable commensurated subgroups, and therefore do not satisfy the assumptions of \cite{BFS}.

Our study of the family of groups $G(F,F')$ provides examples of finitely generated simple groups having non-discrete simple envelopes. As far as we know, the existence of such groups is original.

\begin{thm} \label{thm-simple-simple-lattice}
There exist non-discrete locally compact groups that are compactly generated, abstractly simple, and having (cocompact) lattices that are simple.
\end{thm}

\subsection{Relative commensurators}

If $G$ is a profinite group, the group of abstract commensurators of $G$ consists of equivalence classes of isomorphisms between open subgroups of $G$, where two isomorphisms are identified if they coincide on some open subgroup. The idea of studying abstract commensurators of profinite groups was initiated in \cite{BEW}, with the motivation to use them as a tool to study totally disconnected locally compact groups.

In \cite{germs} the authors proved that when $F$ is 2-transitive and every point stabilizer $F_a$ in $F$ is equal to its normalizer in $\Sy$, the group of abstract commensurators of any compact open subgroup of $U(F)$ is a certain group of almost automorphisms $\mathrm{AAut}_{F_a}(T_{d,2})$ of the quasi-regular rooted tree $T_{d,2}$ (see Theorem C and Theorem 6.14 in \cite{germs}). 

Given a profinite group $G$ and a group $L$ containing $G$, a relative commensurator of $G$ in $L$ is an element of $L$ whose conjugation induces an isomorphism between two open subgroups of $G$. The group of relative commensurators of $G$ in $L$ is denoted $\mathrm{Comm}_L(G)$. 

In Section \ref{sec-further-prop} we give a second interpretation of the group $G(F,F')$ by investigating the relative commensurator of a compact open subgroup $K$ of $U(F)$ in $U(F')$. We prove that, although $\mathrm{Comm}_{U(F')}(K)$ is not equal to $G(F,F')$ in general, we have the following result.

\begin{prop}
Let $d \geq 3$, and let $F \leq F' \leq \Sy$ be two permutation groups such that $F'$ stabilizes the orbits of $F$. Assume that for every $a \in \Omega$, the point stabilizer $F_a$ is equal to its normalizer in $F_a'$. Then $G(F,F')$ is equal to the group of relative commensurators of any compact open subgroup of $U(F)$ in $U(F')$.
\end{prop}

\subsection{Proper action on a CAT(0) cube complex}

When $F$ is strictly contained in $F'$, the action of $G(F,F')$ on $\treed$ is continuous but not proper, and actually the group $G(F,F')$ cannot act continuously and properly on a tree (see Lemma \ref{lem-group-tree-compact}). A tree being nothing but a one dimensional CAT(0) cube complex, this naturally raises the question whether $G(F,F')$ can act continuously and properly on a CAT(0) cube complex. We answer this question in the positive in Section \ref{sec-ccc0}.

\begin{thm} \label{thm-intro-ccc0}
Let $d \geq 3$, and let $F \leq F' \leq \Sy$ be two permutation groups such that $F$ is transitive. Then the group $G(F,F')$ admits a continuous and proper action on a CAT(0) cube complex. 
\end{thm}

This result implies in particular that the group $G(F,F')$ has the Haagerup property. The action of $G(F,F')$ on this CAT(0) cube complex is not cocompact, and actually the group $G(F,F')$ cannot act properly and cocompactly on any CAT(0) metric space (see Remark \ref{rmq-not-cocompact}). This CAT(0) cube complex is not even finite dimensional, and we show that the group $G(F,F')$ cannot act properly on a \textit{finite dimensional} CAT(0) cube complex. 

\subsection*{Organization of the paper}

In Section \ref{sec-notation} we set some notation and terminology, and we establish preliminary results on the groups $G(F,F')$ in Section \ref{sec-preliminaries}. The question of the virtual simplicity of $G(F,F')$ is addressed in Section \ref{sec-simple}, which contains the proofs of Theorem \ref{thm-intro-ind2} and Theorem \ref{thm-intro-ind8}. In Section \ref{sec-further-prop} we investigate further properties of the groups $G(F,F')$, among which the connections with groups of relative commensurators. In Section \ref{sec-ccc0} we give the proof of Theorem \ref{thm-intro-ccc0} by adopting the point of view of commensurating actions. Finally Section \ref{sec-lattices} concerns the study of lattices in the groups $G(F,F')$. We give a concrete criterion to detect the absence of lattices in a locally compact group, and apply it to some locally elliptic groups (see Theorem \ref{thm-L-no-lattice}) and to the groups $G(F,F')$ (see Corollary \ref{cor-without-lattice}).

\subsection*{Acknowledgments}

I am very grateful to Yves de Cornulier for useful remarks and valuable discussions concerning this work. I am also extremely grateful to the authors of \cite{BCGM} for pointing out to my attention the idea of relaxing the local action, and especially to Pierre-Emmanuel Caprace for several comments that largely improved the contents of the paper. Finally I also thank the referee for his corrections and helpful comments improving the exposition.

\section{Notation and terminology} \label{sec-notation}

\subsection{Groups acting on trees}

We will denote by $\Omega$ a set of cardinality $d \geq 3$ and by $\treed$ a regular tree of degree $d$. The vertex set of $\treed$ will be denoted $\verti$ and the set of non-oriented edges will be denoted $\edg$. 

We fix once and for all a coloring $c: \edg \rightarrow \Omega$ such that for every vertex $v \in \verti$, the map $c$ restricts to a bijection $c_v$ from the set $\mathsf{E}(v)$ of edges containing $v$ to $\Omega$. We will refer to $c(e)$ as the \textit{color} of the edge $e$. For every $g \in \autd$ and every $v \in \verti$, the automorphism $g$ induces a bijection $g_v : \mathsf{E}(v) \rightarrow \mathsf{E}(gv)$, which gives rise to a permutation $\sg \in \Sy$ defined by $\sg = c_{gv} \circ g_v \circ c_v^{-1}$. The permutation $\sg$ will be called the \textit{local permutation} of $g$ at the vertex $v$. These permutations satisfy the rules \begin{equation} \label{eq-rules} \sigma(gh,v) = \sigma(g,hv) \sigma(h,v) \, \, \text{and} \, \, \sigma(g^{-1},v) = \sigma(g,g^{-1}v)^{-1} \end{equation} for every $g,h \in \autd$ and $v \in \verti$. 

We easily see that an automorphism $g \in \autd$ is uniquely determined by the image of some vertex together with the collection of permutations $\sg$, where $v \in \verti$. Note that given $\sigma \in \Sy$, there always exists $g \in \autd$ such that all the local permutations of $g$ are equal to $\sigma$, and moreover $g$ may be chosen to be hyperbolic. This observation will be used repeatedly in the paper. 

\bigskip


A vertex $v$ of a subtree $T$ of $\treed$ is called a \textit{leaf} of $T$ if $v$ has exactly one neighbour in $T$, and otherwise $v$ is called an \textit{internal vertex} of $T$. A subtree $T$ of $\treed$ is said to be \textit{complete} if for every internal vertex $v$, all the neighbours of $v$ in $\treed$ belong to $T$. 

For every vertex $v$ and every $n \geq 0$, we will denote by $\mathcal{B}(v,n)$ the subtree of $\treed$ spanned by vertices at distance at most $n$ from $v$. Note that $\mathcal{B}(v,n)$ is a complete subtree as soon as $n \geq 1$.

For every subtree $T$ of $\treed$ and every group $G$ acting on $\treed$, we denote by $G_{T}$ the pointwise stabilizer of $T$ in $G$. For example if $T=e$ is a single edge, then $G_e$ is the subgroup of $G$ fixing both vertices of $e$. The subgroup of $G$ generated by the subgroups $G_e$, where $e$ ranges over the set of edges of $\treed$, will be denoted $G^+$. Note that $G^+$ is a normal subgroup of $G$, and if $G$ is endowed with the topology induced from $\autd$, then $G^+$ is open in $G$.

Recall that the set of vertices $\verti$ admits a natural bipartition, in which two vertices belong to the same block if they are at even distance. The subgroup of $G$ (of index at most two) preserving this bipartition will be called the \textit{type-preserving} subgroup of $G$ and will be denoted $G^{\star}$. Note that the subgroup of $G$ generated by its vertex stabilizers lies inside $G^{\star}$, so a fortiori $G^+$ is also included in $G^{\star}$.

From now and for all the paper we fix an edge $e_0 \in \edg$, whose vertices will be denoted $v_0$ and $v_1$.

\subsection{Permutation groups}

Every partition of $\Omega$ gives rise to a subgroup of $\Sy$ consisting of permutations of $\Omega$ stabilizing each block of the partition. Such a subgroup is called a \textit{Young subgroup} of $\Sy$, and is naturally isomorphic to the direct product of the symmetric groups on each block of the partition. In particular when $F \leq \Sy$ is a permutation group, we can consider the Young subgroup $\hat{F} \leq \Sy$ associated to the partition of $\Omega$ into $F$-orbits. Note that we always have $F \leq \hat{F}$, and $\hat{F} = \Sy$ if and only if the permutation group $F$ is transitive.

Given a permutation group $F \leq \Sy$ and $a \in \Omega$, the stabilizer of $a$ in $F$ will be denoted $F_a$. The (normal) subgroup of $F$ generated by its point stabilizers will be denoted $F^+$. 


\section{Preliminaries} \label{sec-preliminaries}


\subsection{Definitions}

Let us fix a permutation group $F \leq \Sy$. The Burger-Mozes' group $U(F)$ is defined as the subgroup of automorphisms of $\treed$ whose local action is prescribed by $F$ \cite{BM-IHES}, that is \[ U(F) = \left\{g \in \autd \, : \, \sg \in F \, \, \text{for all $v \in \verti$} \right\}. \]

It is a closed subgroup of $\autd$, which is discrete if and only if the permutation group $F$ acts freely on $\Omega$. Clearly $U(F)$ is a subgroup of $U(F')$ when $F \leq F'$. Combined with the fact that the group $U(\left\{1\right\})$ acts transitively on the set $\verti$, this observation implies that $U(F)$ is always vertex-transitive.


\bigskip

The definition of the groups under consideration in this paper can be seen as a relaxation of the definition of the groups $U(F)$, in the sense that the local action is prescribed almost everywhere only. More precisely, we let \[ G(F) = \left\{g \in \autd \, : \, \sg \in F \, \, \text{for all but finitely many $v \in \verti$} \right\}. \]

It readily follows from the multiplication rules (\ref{eq-rules}) that $G(F)$ is a subgroup of $\autd$, and of course one has $U(F) \leq G(F)$.

\begin{defi}
Given $g \in G(F)$, we say that a vertex $v$ is a \textit{singularity} of $g$ if $\sg \notin F$. The set of singularities of $g$ will be denoted $S(g)$.
\end{defi} 

For every $g \in G(F)$, we let $T(g)$ be the $1$-neighbourhood of the subtree of $\treed$ spanned by $S(g)$. Equivalently, $T(g)$ can be defined as the unique minimal complete subtree of $\treed$ such that $\sg \in F$ for every $v \in \verti$ that is not an internal vertex of $T(g)$.

\begin{lem} \label{lem-topo-g(f)}
Let $g \in G(F)$, and denote by $T = T(g)$, $U_T = U(F)_T$ and $U_{g(T)} = U(F)_{g(T)}$. Then one has $g U_{T} g^{-1} = U_{g(T)}$.
\end{lem}

\begin{proof}
Observing that $g(T) = T(g^{-1})$, by symmetry it is enough to prove that $g U_{T} g^{-1} \subset U_{g(T)}$. The fact that $g U_{T} g^{-1}$ fixes pointwise $g(T)$ is easy, so the only thing that needs to be checked is that $g U_{T} g^{-1}$ lies in $U(F)$. So let $u \in U_{T}$ and $v \in \verti$. According to (\ref{eq-rules}), one has \begin{equation} \label{eq-proof-topo} \sigma(gug^{-1},v) = \sigma(g, ug^{-1}v) \, \sigma(u,g^{-1}v) \, \sigma(g,g^{-1}v)^{-1}. \end{equation} As observed previously, the element $gug^{-1}$ fixes pointwise $g(T)$, so we only have to deal with the case when $v$ is not an internal vertex of $g(T)$, i.e.\ when $g^{-1}(v)$ is not an internal vertex of $T$. This implies that $ug^{-1}(v)$ is not an internal vertex of $T$ either, and by definition of $T$ we deduce that $\sigma(g,g^{-1}v)$ and $\sigma(g, ug^{-1}v)$ both belong to $F$. Now $\sigma(u,g^{-1}v)$ belongs to $F$ as well since $u \in U(F)$, so it follows from (\ref{eq-proof-topo}) that $\sigma(gug^{-1},v) \in F$.
\end{proof}

Lemma \ref{lem-topo-g(f)} implies in particular that $G(F)$ commensurates the compact open subgroups of $U(F)$, and it follows (see for instance \cite[Chapter 3]{Bourb-topo}) that there exists a group topology on $G(F)$ such that the inclusion of $U(F)$ in $G(F)$ is continuous and open. In particular the group $G(F)$ is a totally disconnected locally compact group, which is discrete if and only if $F$ acts freely on $\Omega$. We point out that in general $G(F)$ need not be closed in $\autd$ (see Proposition \ref{prop-closure-G(F)}), and the topology on $G(F)$ is \textit{not} the topology induced from $\autd$.

\bigskip

Let $v \in \verti$ being fixed. For every $n \geq 0$, we denote by $K_n(v)$ the set of automorphisms $g \in G(F)$ fixing the vertex $v$ and having all their singularities in $\mathcal{B}(v,n)$. Again, it follows from (\ref{eq-rules}) that $K_n(v)$ is a subgroup of $G(F)$. 
Note that the stabilizer of the vertex $v$ in $G(F)$ is exactly the increasing union \[ G(F)_v = \bigcup_{n \geq 0}^{\nearrow} K_n(v). \] Since the ball $\mathcal{B}(v,n)$ contains finitely many vertices, each $K_n(v)$ contains the stabilizer of the vertex $v$ in $U(F)$ as a finite index subgroup. The latter being compact open, $K_n(v)$ is a compact open subgroup of $G(F)$. Therefore $G(F)_v$ is a locally elliptic open subgroup of $G(F)$, i.e.\ an increasing union of compact open subgroups.

\subsection{Preliminary results}

The following result shows that, although elements of $G(F)$ are not required to act locally like $F$ everywhere, their local action exhibits some rigidity. 

\begin{lem} \label{lem-pres-orbit}
For every $g \in G(F)$ and every vertex $v \in \verti$, the permutation $\sg$ stabilizes the orbits of $F$ in $\Omega$. In other words, the group $G(F)$ is contained in $U(\hat{F})$.
\end{lem}

\begin{proof}
For a given $g \in G(F)$, we consider the set $V_g$ of vertices for which the conclusion does not hold. We want to prove that $V_g$ is empty. The key observation is that if $v$ belongs to $V_g$, then $v$ must have at least two neighbours that also belong to $V_g$. It follows that if $V_g$ is not empty, then it must contain an infinite subtree, which is impossible by definition of $G(F)$.
\end{proof}

For every permutation group $F' \leq \Sy$ such that $F \leq F' \leq \hat{F}$, we denote by $G(F,F')$ the subgroup of $G(F)$ consisting of elements $g \in G(F)$ such that $\sg \in F'$ for all $v \in \verti$, i.e.\ $G(F,F') = G(F) \cap U(F')$. This is the subgroup of $G(F)$ consisting of elements having all their singularities in $F'$. Clearly we have $G(F,F') \leq G(F,F'')$ as soon as $F' \leq F''$, and $G(F,F) = U(F)$ and $G(F,\hat{F}) = G(F)$. Therefore the family of subgroups $G(F,F') \leq G(F)$ interpolates between $U(F)$ and $G(F)$ when $F'$ ranges over subgroups of $\hat{F}$ containing $F$. Note that $G(F,F')$ is always an open subgroup of $G(F)$, and when referring to a topology on $G(F,F')$ we will always mean the induced topology from $G(F)$.

\bigskip

\begin{center}
\textit{From now and for all the paper, we denote by $F,F' \leq \Sy$ two permutation groups such that $F \leq F' \leq \hat{F}$.}
\end{center}

\bigskip

In some sense, the following result can be seen as a converse of Lemma \ref{lem-pres-orbit}. 

\begin{lem} \label{lem-extend-G(F)}
Let $v \in \verti$ and $n \geq 0$. If $h \in \autd$ is such that $\sigma(h,w) \in F'$ for every vertex $w$ in $\mathcal{B}(v,n)$, then there exists $g \in G(F,F')$ such that $g$ and $h$ coincide on $\mathcal{B}(v,n+1)$ and $\sigma(g,w) \in F$ for every vertex $w$ that is not in $\mathcal{B}(v,n)$.
\end{lem}

\begin{proof}
We denote by $\mathcal{S}(v,n)$ the set of vertices which are at distance exactly $n$ from the vertex $v$. For every $x \in \mathcal{S}(v,n)$, we denote by $V_x$ the set of vertices $w$ such that the unique path between $v$ and $w$ contains the vertex $x$. 

Since the group $U(F)$ acts transitively on the set of vertices of $\treed$, we may assume that $h$ fixes the vertex $v$. So we impose that $g$ fixes $v$ as well, and therefore giving the value of $\sigma(g,w)$ for every vertex $w$ is enough to define the element $g$. Naturally we put $\sigma(g,w) = \sigma(h,w)$ for every vertex $w$ in $\mathcal{B}(v,n)$. This implies that $g$ and $h$ coincide on $\mathcal{B}(v,n+1)$, and we must explain how to extend the definition of $g$ to an element of $G(F,F')$.

For every $x \in \mathcal{S}(v,n)$ and every $a \in \Omega$, we choose $\sigma_{a,x} \in F$ such that $\sigma(h,x)(a) = \sigma_{a,x}(a)$. Note that such an element $\sigma_{a,x}$ exists because $\sigma(h,x) \in F' \leq \hat{F}$. Now for every vertex $w \in V_x$ different from $x$, we set $\sigma(g,w) = \sigma_{a(w),x}$, where $a(w)$ is the color of the unique edge emanating from $x$ and separating $x$ and $w$. By construction the definition of the element $g$ is consistent, and $g \in G(F,F')$ because $S(g) \subset \mathcal{B}(v,n)$.
\end{proof}

Recalling that a basis of neighbourhoods for the topology on the group $\autd$ is given by pointwise stabilizers of finite sets, we immediately deduce the following result.

\begin{prop} \label{prop-closure-G(F)}
The closure of $G(F,F')$ in the topological group $\autd$ is the group $U(F')$.

In particular when $F' = \Sy$, the group $G(F)$ is dense in $\autd$ if and only if the permutation group $F$ is transitive.
\end{prop}

We derive the following result, which says in particular that the action of $G(F,F')$ on $\treed$ is never proper when $F$ is strictly contained in $F'$.

\begin{cor} \label{cor-u=g}
The following statements are equivalent:
\begin{enumerate}[label=(\roman*)]
\item $F=F'$;
\item $G(F,F')=U(F)$;
\item $G(F,F')$ is a closed subgroup of $\autd$;
\item vertex stabilizers $G(F,F')_v$ are compact.
\end{enumerate}
\end{cor}

\begin{proof}
The implications $(i) \Rightarrow (ii) \Rightarrow (iii) \Rightarrow (iv)$ are trivial.

$(iv) \Rightarrow (iii)$ follows from a general argument: since $G(F,F')$ is locally compact and its action on $\treed$ is continuous and proper, the subgroup $G(F,F')$ must be closed in $\autd$.

$(iii) \Rightarrow (i)$. Since $G(F,F')$ is closed, according to Proposition \ref{prop-closure-G(F)} the group $G(F,F')$ must contain $U(F')$, and this easily implies that $F = F'$. 
\end{proof}

\subsection{Generators}

For every $n \geq 0$ and every vertex $v \in \verti$, we denote by $K_{n,F'}(v)$ the intersection between $K_n(v)$ and $G(F,F')$. This is the open subgroup of $G(F,F')$ consisting of elements fixing $v$ and having all their singularities in the ball or radius $n$ around $v$.

\begin{prop} \label{prop-gen-KU}
Let $k \geq 0$ and $g \in G(F,F')$ with at most $k$ singularities. Then there exist vertices $v_1, \ldots, v_k \in \verti$ and elements $\gamma \in U(F)$ and $g_i \in K_{0,F'}(v_i)$ such that $g = \gamma g_1 \cdots g_k$.

In particular the group $G(F,F')$ is generated by $U(F)$ together with $K_{0,F'}(v_0)$.
\end{prop}

\begin{proof}
We argue by induction on the number $k$. The result is clear when $k=0$ by definition. Now let $g \in G(F,F')$ having at most $k+1$ singularities, and let $v \in S(g)$. Since the group $U(F)$ is transitive on the set of vertices, there exists $\gamma_1 \in U(F)$ such that $g' = \gamma_1 g$ fixes $v$. Note that since $\gamma_1 \in U(F)$, for every vertex $w$ we have $\sigma(g,w) \notin F$ if and only if $\sigma(g',w) \notin F$. According to Lemma \ref{lem-extend-G(F)} applied with $n=0$, there exists $g_v \in K_{0,F'}(v)$ acting like $g'$ on the star around the vertex $v$. Let $g'' = g' g_v^{-1} = \gamma_1 g g_v^{-1}$. By construction $g''$ fixes the star around $v$, and the singularities of $g''$ are exactly the vertices $g_v(w)$ where $w$ a singularity of $g'$ different from $v$. Therefore $g''$ has at most $k$ singularities, so by the induction hypothesis there exist $v_1, \ldots, v_k \in \verti$ and $\gamma_2 \in U(F)$, $g_{v_i} \in K_{0,F'}(v_i)$, such that $g'' = \gamma_2 g_{v_1} \cdots g_{v_k}$, which can be rewritten $g = (\gamma_1^{-1} \gamma_2) g_{v_1} \cdots g_{v_k} g_{v}$.

Since the group $U(F)$ is vertex-transitive, the subgroup of $G(F,F')$ generated by $U(F)$ and $K_{0,F'}(v_0)$ contains all the subgroups $K_{0,F'}(v)$, for $v\in \verti$. Since all these subgroups together with $U(F)$ generate the group $G(F,F')$ according to the previous paragraph, this proves the second statement.
\end{proof}

Since the group $U(F)$ is always compactly generated and $K_{0,F'}(v_0)$ is compact, we deduce the following. 

\begin{cor} \label{cor-g-ce}
The group $G(F,F')$ is compactly generated. In particular when $F$ acts freely on $\Omega$, the group $G(F,F')$ is finitely generated.
\end{cor}


\begin{rmq}
Since the group $U(F)$ is unimodular and open in $G(F,F')$, the modular function $\Delta$ of $G(F,F')$ must vanish on $U(F)$. Moreover by continuity $\Delta$ vanishes on any compact subgroup as well. Since $G(F,F')$ is generated by $U(F)$ and some compact open subgroup, the group $G(F,F')$ is always unimodular.
\end{rmq}

We end this paragraph by giving a particular compact generating subset when the permutation group $F$ is assumed to be transitive.

\begin{cor} \label{cor-gen-trans}
Assume that $F$ is transitive. Then $G(F,F')^{\star}$ is generated by $K_{0,F'}(v_0)$ and $K_{0,F'}(v_1)$.
\end{cor}

\begin{proof}
Write $S = K_{0,F'}(v_0) \cup K_{0,F'}(v_1)$. Since $F$ is transitive, the group $U(F)^{\star}$ is generated by $U(F)_{v_0} \cup U(F)_{v_1}$, and therefore $U(F)^{\star}$ lies inside the subgroup generated by $S$. So by conjugating the two subgroups $K_{0,F'}(v_0)$ and $K_{0,F'}(v_1)$ one may obtain $K_{0,F'}(v)$ in $\left\langle S\right\rangle$ for every $v \in \verti$, and we conclude the corollary thanks to Proposition \ref{prop-gen-KU}.
\end{proof}

We derive from Corollary \ref{cor-gen-trans} the following result. 

\begin{prop} \label{prop-morph-trivial}
Assume that $F$ acts simply transitively on $\Omega$. Assume also that $k < d$ is such that any action of $F'$ on a set of cardinality $k$ is trivial. Then any morphism $\varphi: G(F,F')^{\star} \rightarrow \mathrm{Aut}(\mathcal{T}_{k})$ is trivial.  
\end{prop}

\begin{proof}
We fix $v \in \verti$, and we let $K = K_{0,F'}(v)$. The map $K \rightarrow F'$, $g \mapsto \sigma(g,v)$, is a group morphism by (\ref{eq-rules}), which is onto according to Lemma \ref{lem-extend-G(F)}. Moreover it is also injective since $F$ acts freely on $\Omega$, so that the subgroup $K$ is isomorphic to $F'$. The assumption on the group $F'$ implies that it does not have any subgroup of index two, and therefore the image $\varphi(K)$, which is a finite subgroup of $\mathrm{Aut}(\mathcal{T}_{k})$, must fix a vertex $w$ of $\mathcal{T}_{k}$. The action of $\varphi(K)$ on the set of edges around $w$ yields an action of $F'$ on a set of cardinality $k$, which must be trivial by assumption. Therefore $\varphi(K)$ actually fixes the star around $w$, and it follows that $\varphi(K)$ has to be trivial.

Now since the action of $F$ on $\Omega$ is transitive, by Corollary \ref{cor-gen-trans} the group $G(F,F')^{\star}$ is generated by $K_{0,F'}(v_0)$ and $K_{0,F'}(v_1)$. According to the previous paragraph both must be sent by $\varphi$ to the identity, and it follows that $\varphi$ is trivial.
\end{proof}

\section{Simplicity} \label{sec-simple}

Recall that Tits introduced in \cite{Tits} a simplicity criterion for groups acting on trees, usually referred to as Tits' independence property (P). The groups $G(F,F')$ do not satisfy Tits' independence property (P), but rather a weaker independence property that we will call the \textit{edge-independence property}. Following the same strategy as in the proof of Tits' theorem, we establish in the following paragraph a simplicity result based on the edge-independence property (see Corollary \ref{cor-with-ind}). The second part of this section will be devoted to the application to the groups $G(F,F')$.

\subsection{A simplicity criterion}

Recall that if $T$ is a simplicial tree, we say that the action of a group $G$ on $T$ is \textit{minimal} if $G$ does not stabilize any proper subtree of $T$. If $T'$ is a subtree of $T$, we denote by $G_{T'}$ the pointwise stabilizer of $T'$ in $G$. We also let $G^+$ be the subgroup of $G$ generated by the set of subgroups $G_e$, where $e$ ranges over the set of edges of $T$.

If $e$ is an edge of $T$ and $v$ a vertex of $e$, we denote by $T_e(v)$ the subtree of $T$ spanned by vertices whose closest point projection on the edge $e$ is the vertex $v$. A subtree $T'$ of $T$ is called a \textit{half-tree} if $T' = T_e(v)$ for some edge $e$ and vertex $v$.

\bigskip

The main result of \cite{Tits} says that if $G$ satisfies Tits' independence property (P) (a definition of which can be found in \cite{Tits}) and acts minimally on $T$ without fixing any end of $T$, then the group $G^+$ is simple as soon as it is not trivial. This remarkable result has been extensively used to establish simplicity of various groups. For example the group $U(F)^+$ is simple as soon as the permutation group $F$ does not act freely on $\Omega$. 

The goal of this paragraph is to prove a simplicity criterion, namely Corollary \ref{cor-with-ind}, by weakening the assumption that the group satisfies Tits' independence property (P). Our motivation comes from the fact that the groups $G(F,F')$ do not satisfy Tits' independence property (P) as soon as $F$ is a proper subgroup of $F'$. 

\bigskip

Let $G \leq \mathrm{Aut}(T)$. Given an edge $e$ of $T$, we denote by $T'$ and $T''$ the two half-trees separated by $e$. The group $G_{e}$ induces permutation groups $G_{e}'$ and $G_{e}''$ on the set of vertices of $T'$ and $T''$, so that we have a natural injective homomorphism $\varphi_{e}: G_{e} \rightarrow G_{e}' \times G_{e}''$.

\begin{defi}
We say that a group $G \leq \mathrm{Aut}(T)$ satisfies the \textit{edge-independence property} if $\varphi_{e}$ is an isomorphism for every choice of $e$.
\end{defi}

This means that for every edge $e$, the pointwise stabilizer of $e$ in $G$ acts independently on the two half-trees emanating from $e$. The edge-independence property already appeared in \cite{Amann, Ba-El-Wi} (it is called \enquote{independence property} in \cite{Amann} and \enquote{property $IP_1$} in \cite{Ba-El-Wi}), and is strictly weaker than Tits' independence property (P). However, one can check that these are equivalent for closed subgroups of $\mathrm{Aut}(T)$ (see for instance \cite[Lemma 10]{Amann}). Note that when $G$ satisfies the edge-independence property, the subgroup $G^+$ is also the subgroup generated by pointwise stabilizers of half-trees in $G$. 

\bigskip 

The following two lemmas are standard. Recall that the action of a group $G$ on $T$ is said to be of \textit{general type} if there exist in $G$ hyperbolic isometries without common endpoints, or equivalently if $G$ does not have any finite orbit in $T \cup \partial T$.

\begin{lem} \label{lem-g+-tg}
Let $H$ be a non-trivial subgroup of $\mathrm{Aut}(T)$, and $G \leq \mathrm{Aut}(T)$ a subgroup normalizing $H$. If the action of $G$ on $T$ is minimal and of general type, then the same holds for $H$.
\end{lem}

\begin{proof}
Since the set of fixed vertices of $H$ in $T$ is $G$-invariant, if it is non-empty then by minimality of the action of $G$ it must be the entire $T$, which is a contradiction with $H \neq 1$. For the same reason, we see that $H$ does not stabilize an edge, and $H$ cannot stabilize line because this would contradict the existence of independent hyperbolic elements in $G$. Moreover if $H$ has a unique fixed boundary point, then this is also a fixed point for $G$, which is again a contradiction. So it follows that the action of $H$ on $T$ must be of general type. In particular $H$ contains hyperbolic isometries, and it follows that $H$ admits a unique invariant minimal subtree $T'$. By uniqueness, $T'$ must be $G$-invariant, so by minimality of the action of $G$ we must have $T'=T$.
\end{proof}

\begin{lem} \label{lem-axis-half-tree}
Let $G$ be a subgroup of $\mathrm{Aut}(T)$ whose action on $T$ is minimal and of general type. Given any half-tree $T' \subset T$, there exists a hyperbolic element in $G$ whose axis is contained in $T'$.
\end{lem}

\begin{proof}
Let $X$ be the set of hyperbolic elements of $G$. First remark that there exists $x \in X$ having an endpoint in $\partial T'$. Indeed, otherwise we would have a $G$-invariant subtree (namely the union of the axes of the elements of $X$) contained in the complement of $T'$, which is a contradiction with the fact that $G$ acts minimally on $T$. Now the conclusion follows from the fact that if $y \in X$ does not have any endpoint in common with $x$, then there exists some integer $k \in \mathbb{Z}$ such that the axis of $x^k y x^{-k}$ is contained in $T'$.
\end{proof}

The following result plays an essential role in the proof of Theorem \ref{thm-without-ind}. In the proof we make use of the classical idea of using double commutators, which appears for example in \cite[Theorem 4]{new-hor}.

\begin{lem} \label{lem-game-commutator}
Let $G$ be a subgroup of $\mathrm{Aut}(T)$, and $N$ a subgroup of $\mathrm{Aut}(T)$ normalized by $G$. Let $T'$ be a half-tree in $T$. Assume that $N$ contains a hyperbolic element whose axis is contained in $T'$. Then $N$ contains the derived subgroup of $G_{T'}$.
\end{lem}

\begin{proof}
Let $\gamma \in N$ be a hyperbolic element whose axis is contained in $T'$. We let $e$ be the edge of $T$ and $v$ the vertex of $T$ such that $T'$ is the half-tree emanating from $e$ containing $v$. We denote by $w$ the projection of the vertex $v$ on the axis of $\gamma$. We denote by $L$ the maximal subtree of $T$ containing $w$ but not its neighbours on the axis of $\gamma$. By construction the subtrees $L$ and $\gamma(L)$ are disjoint, and $\gamma^{\pm 1}(L) \subset T'$. This implies that for every $g \in G_{T'}$, the element $[g,\gamma] = g \gamma g^{-1} \gamma^{-1} \in N$ acts like $g$ on $L$, like $\gamma g^{-1} \gamma^{-1}$ on $\gamma(L)$, and is the identity elsewhere. It follows that for every $h \in G_{T'}$, the element $[[g,\gamma],h]$ (which remains in $N$) acts like $[g,h]$ on $L$ and is the identity elsewhere, and therefore this element is equal to $[g,h]$.
\end{proof}

\begin{thm} \label{thm-without-ind}
Suppose that $G \leq \mathrm{Aut}(T)$ acts minimally on $T$ and does not fix any end. If $N$ is a non-trivial subgroup of $\mathrm{Aut}(T)$ normalized by $G^+$, then $N$ contains the derived subgroup of $G_{T'}$ for every half-tree $T'$.
\end{thm}

\begin{proof}
We may assume that $G^+$ is non-trivial, which implies in particular that $T$ is neither a point nor a line. By assumption $G$ acts minimally on $T$ and does not fix any end, so it follows that $G$ is of general type. Thanks to Lemma \ref{lem-g+-tg} applied successively to $G^+$ and $N$, we deduce that the action of $N$ on $T$ is minimal and of general type. Therefore we are in position to apply Lemma \ref{lem-axis-half-tree}, which ensures the existence of a hyperbolic element of $N$ whose axis is contained in $T'$. The fact that the derived subgroup of $G_{T'}$ is contained in $N$ then follows from Lemma \ref{lem-game-commutator}.
\end{proof}

The difference between the following result and Tits' theorem \cite{Tits} is that the independence assumption is strictly weaker here. This is counterbalanced by the fact that we impose a condition on pointwise stabilizers of edges in order to obtain simplicity of $G^+$.

\begin{cor} \label{cor-with-ind}
Let $G$ be a subgroup of $\mathrm{Aut}(T)$ such that:
\begin{enumerate}[label=(\alph*)]
\item $G$ acts minimally on $T$ and does not fix any end of $T$;
\item $G$ satisfies the edge-independence property.
\end{enumerate}
Assume that $N$ is a non-trivial subgroup of $\mathrm{Aut}(T)$ normalized by $G^+$. Then $N$ contains $[G_e,G_e]$ for every edge $e$.

In particular if all $G_e$ are perfect groups, then $G^+$ is simple (or trivial). 
\end{cor}

\begin{proof}
We denote by $T'$ and $T''$ the two half-trees emanating from the edge $e$. It follows from the assumption that $G$ satisfies the edge-independence property that $G_e$ is equal to the product of the subgroups $G_{T'}$ and $G_{T''}$, so that $[G_e,G_e]$ is the product of the derived subgroups of $G_{T'}$ and $G_{T''}$. Now according to Theorem \ref{thm-without-ind}, the derived subgroups of $G_{T'}$ and $G_{T''}$ are contained in $N$, so it follows that $[G_e,G_e]$ is also contained in $N$. This proves the statement. 
\end{proof}

\subsection{Application to the groups $G(F,F')$}

In this subsection, we isolate sufficient conditions on $F \leq F'$ so that the group $G(F,F')$ has a simple subgroup of finite index.

Remark that the group $G(F,F')$ always satisfies the assumptions of Corollary \ref{cor-with-ind}, so that by applying directly the second statement in Corollary \ref{cor-with-ind} we deduce that $G(F,F')^+$ is simple as soon as all the $G(F,F')_e$ are perfect and non-trivial. 
However by refining the argument rather than applying directly the second statement in Corollary \ref{cor-with-ind}, we will establish simplicity results under much more general assumptions, namely Theorem \ref{thm-simple-ind2} and Theorem \ref{thm-simple-ind8}.

\bigskip

Recall that for $G \leq \autd$, we denote by $G^{\star}$ the subgroup (of index at most two) of $G$ preserving the natural bipartition of the set of vertices. Clearly $G^{\star}$ contains $G^+$ as a normal subgroup.

The following proposition characterizes permutation groups $F,F'$ for which $G(F,F')^+ = G(F,F')^{\star}$. Note that this result implies that if the group $G(F,F')$ is virtually simple, then $F$ and $F'$ must satisfy the condition $(iii)$. In the case of the group $G(F)$, i.e.\ when $F' = \Sy$, this condition becomes that $F$ must be transitive (note the difference with the group $U(F)$ \cite[Proposition 3.2.1]{BM-IHES}).

\begin{prop} \label{prop-g(f)simple}
The following conditions are equivalent:
\begin{enumerate}[label=(\roman*)]
\item $G(F,F')^+$ has index two in $G(F,F')$, that is $G(F,F')^+ = G(F,F')^{\star}$;
\item $G(F,F')^+$ has finite index in $G(F,F')$;
\item $F$ is transitive and $F'$ is generated by its points stabilizers.
\end{enumerate}
\end{prop}

\begin{proof}
$(ii) \Rightarrow (iii)$. We let $\Omega_1, \ldots, \Omega_r$ be the orbits of $F$ in $\Omega$. For every $a \in \Omega$, we let $w(a)$ be the unique integer such that $a \in \Omega_{w(a)}$. We identify the tree $\treed$ with the Cayley graph of $\Gamma = \left\langle x_a, \, a \in \Omega \, | \, x_a^2 = 1\right\rangle$. Let us consider the quotient  $\Gamma_F$ of $\Gamma$ defined by adding the relation $x_a = x_b$ when $w(a)=w(b)$. The Cayley graph $\mathcal{T}_F$ of $\Gamma_F$ is a regular tree of degree $r$, and we have a natural projection $p_F: \treed \rightarrow \mathcal{T}_F$.

Let $g \in G(F,F')$ fixing some vertex, and let $v,v' \in \verti$ such that $v' = g(v)$. Since $g$ fixes a vertex, the distance between $v$ and $v'$ must be even. Let us consider the unique path from $v$ to $v'$, whose sequence of colors of edges is denoted by $(a_1, \ldots, a_{2n})$. Since the element $g$ fixes a vertex and stabilizes the orbits of $F$ on the set of edges, the sequence $w(a_1), \ldots, w(a_{2n})$ is palindromic, which implies that the vertices $v$ and $v'$ have the same image by the projection $p_F$. In other words, any $g \in G(F,F')$ fixing some vertex must stabilize the fibers of vertices of the map $p_F$, and a fortiori the same holds for the group $G(F,F')^+$. Now if $r \neq 1$ then the tree $\mathcal{T}_F$ is infinite. Therefore $G(F,F')^+$ must have infinitely many orbits of vertices in $\treed$, which prevents $G(F,F')^+$ from being of finite index in $G(F,F')$.

Now we want to prove that $F'$ is generated by its point stabilizers, or equivalently that $F'^+$ is transitive on $\Omega$. We carry out the same construction as in the previous paragraph by replacing $F$-orbits by $F'^+$-orbits. Since any element of $G(F,F')^+$ has all its local permutations in $F'^+$, the group $G(F,F')^+$ must stabilize the fibers of the projection, and the conclusion follows by the same argument.

$(iii) \Rightarrow (i)$. The fact that $F'$ is transitive and generated by its point stabilizers implies that for every vertex $v \in \verti$, the group $G(F,F')^+_v$ is transitive on the set of edges around $v$. We easily deduce that $G(F,F')^+$ is transitive on the set of non-oriented edges of $\treed$, and therefore has index two in $G(F,F')$.
\end{proof}

We denote by $N(F,F')$ the subgroup of $G(F,F')^+$ generated by the derived subgroups of pointwise stabilizers of edges, that is \[ N(F,F') = \left\langle [G(F,F')_e,G(F,F')_e]\right\rangle,\] where $e$ ranges over $\edg$. Clearly $N(F,F')$ is normal in $G(F,F')$. Note that when $F'$ acts freely on $\Omega$, all the $G(F,F')_e$ are trivial, so that $N(F,F')$ is trivial as well. 

\begin{lem} \label{lem-Gamma-triv}
Let $e\in \edg$, and let $T'$ be one of the two half-trees defined by $e$. Write $a=c(e)$, and assume that the point stabilizer $F_a'$ is non-trivial. Then $G(F,F')_{T'}$ is not solvable.
\end{lem}

\begin{proof}
We write $H = G(F,F')_{T'}$, and we prove that $H$ contains a copy of itself in its derived subgroup. Combined with the fact that $H$ is non-trivial, this implies the statement.

We let $v$ be the vertex of $e$ not contained in $T'$, and we consider the vertex $w$ at distance two from $v$ and such that the unique path $(e',e'')$ between $v$ and $w$ is colored $(b,a)$. If we denote by $T''$ the half-tree defined by $e''$ and not containing $w$, then $K = G(F,F')_{T''}$ is conjugate to $H$ inside $G(F,F')$, so in particular $K$ is isomorphic to $H$.

We let $\sigma \in F'$ such that $\sigma(a)=a$ and $\sigma(b) \neq b$, and we fix an element $h \in H$ such that $\sigma(h,v) = \sigma$. If we denote by $X$ the half-tree facing $T''$, then the half-trees $X$ and $h(X)$ are disjoint by construction. For every $g \in K$, we consider the element $[g,h] = g (h g^{-1} h^{-1})$. Since $g$ is supported in $X$, we easily see that $[g,h]$ is supported in $X \, \sqcup \, h(X)$, and that $[g,h]$ is equal to $g$ on $X$ and $[g,h]$ is equal to $h g^{-1} h^{-1}$ on $h(X)$. In particular it follows that the map $\varphi_h: K \rightarrow G(F,F')$, defined by $g \mapsto [g,h]$, is an injective group morphism. Since $K$ lies in $H$, the image of $\varphi_h$ is clearly contained in the derived subgroup of $H$. Therefore the derived subgroup of $H$ contains a copy of $H$, and combined with the fact that $H$ is non-trivial (for instance because $h$ is a non-trivial element of $H$), this shows that $H$ cannot be solvable.
\end{proof}


We derive from Corollary \ref{cor-with-ind} the following structure result for normal subgroups of $G(F,F')^+$.

\begin{cor} \label{cor-monolithe}
Assume that the action of $F'$ on $\Omega$ is not free. Then the following hold:
\begin{enumerate}[label=(\alph*)]
\item $N(F,F')$ is equal to the intersection of all non-trivial normal subgroups of $G(F,F')^+$;
\item $N(F,F')$ is a simple group;
\item $U(F)^+ \leq N(F,F')$, so in particular $N(F,F')$ is open in $G(F,F')$.
\end{enumerate}
\end{cor}

\begin{proof}
$(a)$. Since the group $G(F,F')$ satisfies the assumptions of Corollary \ref{cor-with-ind}, we deduce that the intersection of all non-trivial normal subgroup of $G(F,F')^+$ contains $N(F,F')$. To see that the converse inclusion also holds, remark that $N(F,F')$ is itself a non-trivial (normal) subgroup according to Lemma \ref{lem-Gamma-triv}.

$(b)$. Let us consider the intersection $M$ of all non-trivial normal subgroups of $N(F,F')$. We shall prove that $M = N(F,F')$. First observe that $M$ is a characteristic subgroup of $N(F,F')$. The latter being normal in $G(F,F')^+$, the subgroup $M$ is a normal subgroup of $G(F,F')^+$. So if we prove that $M$ is non-trivial, then $M$ must contain $N(F,F')$ according to statement $(a)$, and we will have $M = N(F,F')$.

Let $T'$ be a half-tree in $\treed$, and $N$ a non-trivial normal subgroup of $N(F,F')$. Since $N(F,F')$ is non-trivial, it follows from Lemma \ref{lem-g+-tg} that the action of $N(F,F')$ on $\treed$ is minimal and of general type. Therefore we may apply Theorem \ref{thm-without-ind}, which shows that $N$ contains the derived subgroup of $N(F,F')_{T'}$. By Lemma \ref{lem-Gamma-triv} the latter is non-trivial, so in particular the subgroup $M$ is non-trivial.

$(c)$. We may clearly assume that $U(F)^+$ is non-trivial, i.e.\ that $F$ does not act freely on $\Omega$. Thanks to Lemma \ref{lem-Gamma-triv} applied with $F=F'$, we see that fixators of half-trees in $U(F)$ are not abelian. In particular, by definition of the subgroup $N(F,F')$, we deduce that $U(F)^+$ and $N(F,F')$ must intersect non-trivially. Since the group $U(F)^+$ is simple by Tits' theorem \cite{Tits}, we must have $U(F)^+ \leq N(F,F')$.
\end{proof}

When $F'$ does not act freely on $\Omega$, Corollary \ref{cor-monolithe} says in particular that $G(F,F')^+$ admits a unique minimal non-trivial normal subgroup, which is open in $G(F,F')^+$. In particular any non-trivial normal subgroup of $G(F,F')^+$ is open, which immediately implies the following result. 

\begin{cor}
The group $G(F,F')^+$ is topologically simple if and only if it is abstractly simple.
\end{cor}

\subsubsection{First theorem}


Given a permutation group $H \leq \Sy$, and a subset $S_a \subset H_a$ for every $a \in \Omega$, we will denote by $\left\langle S_a\right\rangle$ the subgroup of $H$ generated by the $S_a$, where $a$ is implicitly assumed to range over $\Omega$. For example $\left\langle [H_a,H_a]\right\rangle$ is the subgroup of $H$ generated by the derived subgroups of point stabilizers in $H$.


\begin{lem} \label{lem-gamma-uf}
Assume that $\rho$ belongs to $\left\langle [F_a',F_a']\right\rangle$, and let $v \in \verti$. Then there exists $\gamma \in N(F,F')_v$ such that $\sigma(\gamma,v) = \rho$ and $\sigma(\gamma,w) \in F$ for every $w \neq v$. 
\end{lem}

\begin{proof}
By assumption the permutation $\rho$ can be written $\rho = \prod [\alpha_k,\beta_k]$, where for every $k$ the elements $\alpha_k, \beta_k \in F'$ fix a common point $a_k \in \Omega$. According to Lemma \ref{lem-extend-G(F)} applied with $n=0$, for every $k$ one can find elements $g_k, h_k \in G(F,F')_v$ such that $\sigma(g_k,v) = \alpha_k$ and $\sigma(h_k,v) = \beta_k$, and $\sigma(g_k,w),\sigma(h_k,w) \in F$ for every $w \neq v$. Let $\gamma = \prod [g_k, h_k]$. By construction $\gamma$ fixes the vertex $v$ and one has $\sigma(\gamma,v) = \prod [\alpha_k,\beta_k] = \rho$ and $\sigma(\gamma,w) \in F$ for every $w \neq v$. Moreover for every $k$, the elements $g_k, h_k$ fix the edge emanating from $v$ and having color $a_k$. It follows that each $[g_k, h_k]$ belongs to $N(F,F')$, and consequently $\gamma \in N(F,F')$.
\end{proof}

Note that since $\left\langle [F_a',F_a']\right\rangle$ is normal in $F'$, the subgroup $\left\langle [F_a',F_a'] \cup F_a \right\rangle$ of $F'$ generated by the derived subgroups of its point stabilizers together with point stabilizers in $F$ is equal to $\left\langle [F_a',F_a']\right\rangle F^+$. 

\begin{prop} \label{prop-gamma-u(f)}
Assume that $F$ is transitive and that $F$ is contained in $\left\langle [F_a',F_a'] \cup F_a \right\rangle = \left\langle [F_a',F_a']\right\rangle F^+$. Then $N(F,F')$ contains $U(F)^{\star}$.
\end{prop}


\begin{proof}
Observe that since $F$ is transitive, the group $U(F)^{\star}$ is generated by vertex stabilizers in $U(F)$, so it is enough to prove that $U(F)_v$ lies in $N(F,F')$ for every $v \in \verti$.

Let $g \in U(F)_v$, and consider the permutation $\sg$. By assumption one can write $\sg = \sigma \sigma'$ with $\sigma \in \left\langle [F_a',F_a']\right\rangle$ and $\sigma' \in F^+$. Applying Lemma \ref{lem-gamma-uf} to the permutation $\sigma^{-1}$, we see that there exists $\gamma_1 \in N(F,F')$ such that $\gamma_1 g$ remains in $U(F)$ and $\sigma(\gamma_1 g,v) = \sigma' \in F^+$. Therefore one can find $\gamma_2 \in U(F)^+$ such that $\gamma_2 \gamma_1 g$ acts trivially on the star around $v$, and in particular $\gamma_2 \gamma_1 g$ is contained in $U(F)^+$. Now by Corollary \ref{cor-monolithe}, the group $U(F)^+$ is contained in $N(F,F')$, so $\gamma_2 \gamma_1 g \in N(F,F')$ and consequently $g \in N(F,F')$.
\end{proof}

We now prove the main result of this paragraph. Note that by Proposition \ref{prop-g(f)simple}, for the group $G(F,F')^{\star}$ to be simple, it is necessary that $F$ is transitive and that $F'$ is generated by its point stabilizers. To ensure that $G(F,F')^{\star}$ is simple, we slightly strengthen the second assumption by requiring that $F'$ is generated by the derived subgroups of its point stabilizers together with point stabilizers of $F$. 

\begin{thm} \label{thm-simple-ind2}
Let $F \leq F' \leq \Sy$ be permutation groups such that $F$ is transitive, and $F' = \left\langle [F_a',F_a'] \cup F_a \right\rangle$. Then the type-preserving subgroup of $G(F,F')$ is simple.
\end{thm}


\begin{proof}
First note that the assumption on $F'$ implies that $F'$ is generated by its point stabilizers, so that $G(F,F')^{\star}$ is equal to $G(F,F')^+$ according to Proposition \ref{prop-g(f)simple}. Now by Corollary \ref{cor-monolithe}, the group $G(F,F')^+$ is simple if and only if $N(F,F') = G(F,F')^+$. So we let $g$ be an element of $G(F,F')^+$, and we prove that $g \in N(F,F')$.

Argue by induction on the cardinality of $\Sg$. Assume first that $\Sg$ is empty, i.e.\ $g \in U(F)^{\star}$. Since $\left\langle [F_a',F_a'] \cup F_a \right\rangle$ contains $F$, by Proposition \ref{prop-gamma-u(f)} the subgroup $N(F,F')$ contains $U(F)^{\star}$, and therefore $g \in N(F,F')$. Now assume that $\Sg$ has cardinality $n+1 \geq 1$, and let $v \in \Sg$. According to the assumption $F' = \left\langle [F_a',F_a'] \cup F_a \right\rangle$, there exists $\rho \in \left\langle [F_a',F_a']\right\rangle$ such that $\rho \sg$ belongs to $F$. Therefore applying Lemma \ref{lem-gamma-uf} to the vertex $g(v)$ and to the permutation $\rho$, we obtain an element $\gamma \in N(F,F')$ such that $S(\gamma g)$ has cardinality at most $n$. By induction $\gamma g \in N(F,F')$, and finally $g \in N(F,F')$.
\end{proof}


We point out that the class of permutation groups $F,F'$ satisfying the assumptions of Theorem \ref{thm-simple-ind2} is very large. For example it is enough to have $F'$ generated by the derived subgroups of its point stabilizers. Examples of such permutation groups are given by $\mathrm{Alt}(d)$ for $d \geq 5$, or $\mathrm{PSL}(2,q)$ acting on the projective line $\mathbb{P}^1(\mathbb{F}_q)$ for any prime power $q \neq 2,3$. A fortiori it is enough to take for $F'$ any $2$-transitive permutation group with perfect stabilizers (for example simple non-abelian). Examples of such permutation groups can be found in \cite[Example 3.3.1]{BM-IHES}, and we refer the reader to \cite{Dixon} for a list of finite 2-transitive permutation groups. 

When $d=4$, one may check that the only examples of $F \lneq F' \leq \mathrm{Sym}(4)$ satisfying the assumptions of Theorem \ref{thm-simple-ind2} are $F=D_4$ the dihedral group, and $F'=\mathrm{Sym}(4)$.

\bigskip

When specializing to discrete groups, i.e.\ when the permutation group $F$ is moreover assumed to act freely on $\Omega$, we obtain the following result. Note that the assumption implies in particular that $F'$ is a perfect group.

\begin{cor} \label{cor-g(f,f')-simple}
Let $F \leq \Sy$ be a permutation group whose action on $\Omega$ is simply transitive. Assume that $F'$ is generated by the derived subgroups of its point stabilizers. Then the type-preserving subgroup of $G(F,F')$ is a simple group.
\end{cor}

\begin{ex} \label{ex-alt}
When $d \geq 5$ and $F' = \mathrm{Alt}(d)$, a simply transitive subgroup $F$ is for instance given by a cycle of length $d$ if $d$ is odd. If $d=4n$, one can choose for $F$ the abelian subgroup generated by $(1,\ldots,2n) (2n+1,\ldots,4n)$ together with $\prod_{i=1}^{2n} (i,2n+i)$.
\end{ex}

Note that since $\mathrm{Alt}(d)$ is two-generated, Corollary \ref{cor-gen-trans} yields finite generating subsets for the groups $G(F,F')^{\star}$ from Example \ref{ex-alt} consisting of four elements. Note also that since $\mathrm{Alt}(d)$ satisfies the assumption of Proposition \ref{prop-morph-trivial} with $k=d-1$, all these examples $G(F,F')$ are pairwise non-isomorphic when $d$ varies.


\subsubsection{Second theorem}

We begin this paragraph by focusing on the particular case when the permutation group $F$ has index two in $F'$. Under this assumption, we identify a certain canonical subgroup of index eight in $G(F,F')$.

Recall that $e_0 \in \edg$ is a fixed edge whose vertices are denoted $v_0$ and $v_1$. We will denote by $\mathsf{V}_0$ (resp.\ $\mathsf{V}_1$) the set of vertices at even distance from $v_0$ (resp.\ $v_1$). For $g \in G(F,F')$ and $i \in \left\{0,1\right\}$, we denote by $S_i(g)$ the intersection of $S(g)$ with $\mathsf{V}_i$. Clearly $S(g) = S_0(g) \sqcup S_1(g)$. 

The following lemma identifies the set of singularities of the product of two elements in $G(F,F')^{\star}$.

\begin{lem} \label{lem-sing-ind2}
Assume that $(F':F)=2$, and let $i \in \left\{0,1\right\}$. Then $S_i(gh) = S_i(h) \triangle h^{-1}S_i(g)$ for every $g,h \in G(F,F')^{\star}$.
\end{lem}

\begin{proof}
The inclusions \[ S_i(h) \triangle h^{-1} S_i(g) \subset S_i(gh) \subset S_i(h) \cup h^{-1} S_i(g) \] are always satisfied, and follow from (\ref{eq-rules}) together with the fact that $g$ and $h$ preserve the set $\mathsf{V}_i$. To prove the statement, we shall prove that if a vertex $v$ belongs to both $S_i(h)$ and $h^{-1} S_i(g)$, then it is not a singularity of $gh$. Since $v \in S_i(h)$ and $h(v) \in S_i(g)$, then $\sigma(gh,v) = \sigma(g,hv) \sigma(h,v)$ is the product of two elements of $F' \setminus F$, and therefore belongs to $F$ because $F$ has index two in $F'$. So $v \notin S_i(gh)$, and the statement is proved.  
\end{proof}

We are grateful to Nicolas Radu for correcting an earlier version of the argument here.

\begin{prop}
Assume that $(F':F)=2$, and let $i \in \left\{0,1\right\}$. Then \[ G_i(F,F') = \left\{g \in G(F,F')^{\star} \, : \, S_i(g) \, \, \text{has even cardinality} \right\} \] is a subgroup of index two in $G(F,F')^{\star}$.
\end{prop}

\begin{proof}
It follows from Lemma \ref{lem-sing-ind2} that for every $g,h \in G(F,F')^{\star}$, the set $S_i(gh)$ has cardinality \[ \left|S_i(gh)\right| = \left|S_i(h)\right| + \left|S_i(g)\right| - 2 \left|S_i(h) \cap h^{-1}S_i(g)\right|. \] Therefore if $g,h \in G_i(F,F')$, then the product $gh$ remains in $G_i(F,F')$. Moreover for $g \in G(F,F')^{\star}$, we have $S_i(g^{-1}) = g S_i(g)$, so it is clear that $G_i(F,F')$ is stable by inversion. So we have proved that $G_i(F,F')$ is a subgroup of $G(F,F')^{\star}$. The fact that its index is equal to two is clear. 
\end{proof}

We now return to the situation when $F,F'$ only satisfy $F \leq F' \leq \hat{F}$, but we assume that there is a permutation group $F''$ between $F$ and $F'$ and having index two in $F'$. In other words we have $F \leq F'' \leq F' \leq \hat{F}$ and $(F':F'')=2$.

Remark that in this situation $F''$ must contain the derived subgroup of $F'$. Therefore the subgroup $\left\langle [F_a',F_a'] \cup F_a \right\rangle$ is a subgroup of $F''$, so that the assumption $F' = \left\langle [F_a',F_a'] \cup F_a \right\rangle$ of Theorem \ref{thm-simple-ind2} cannot be satisfied.

Since we now have three permutation groups, talking about singularities might be unclear. In order to avoid any ambiguity, we will adopt the following notation for $g \in G(F,F')$: \[\Sigma(g) = \left\{v \in \verti : \sg \notin F'' \right\}. \]




\begin{lem} \label{lem-deux-sing}
Assume that $N(F,F')$ contains $G(F,F'')^{\star}$. Then given any two vertices $v \neq w$ at even distance from each other, there exists $\gamma \in N(F,F')$ such $\Sigma(\gamma) = \left\{v,w\right\}$. 
\end{lem}

\begin{proof}
Let us consider an element $g_1 \in G(F,F')$ fixing $v$ and such that $\Sigma(g_1) = \left\{v\right\}$. We also denote by $g_2$ an element of $G(F,F'')^{\star}$ such that $g_2(v) = g_1^{-1}(w)$ (such an element exists because $v$ and $g_1^{-1}(w)$ remain at even distance from each other). Define $\gamma=[g_2,g_1] = (g_2g_1g_2^{-1}) g_1^{-1}$. Since $g_2 \in G(F,F'')$, the element $g_2g_1g_2^{-1}$ has only one singularity that does not belong to $F''$, namely $g_2(v)$. By applying Lemma \ref{lem-sing-ind2}, we obtain \[\Sigma(\gamma) = \left\{v\right\} \triangle \left\{g_1g_2(v)\right\} = \left\{v,w\right\}.\] Now $g_2$ belongs to $N(F,F')$ because $g_2 \in G(F,F'')^{\star}$ and $G(F,F'')^{\star} \leq N(F,F')$ by assumption. Moreover $N(F,F')$ is normal in $G(F,F')$, so the element $\gamma$ remains in $N(F,F')$, and the proof is complete.
\end{proof}

We are now able to prove the main result of this paragraph. Note that Theorem \ref{thm-simple-ind2} and Theorem \ref{thm-simple-ind8} are complementary, in the sense that examples of permutation groups $F,F''$ satisfying the second assumption of Theorem \ref{thm-simple-ind8} can be found by making use of Theorem \ref{thm-simple-ind2}.


\begin{thm} \label{thm-simple-ind8}
Let $F \leq F'' \leq F'$ be permutation groups such that:
\begin{enumerate} [label=(\alph*)]
	\item $F''$ has index two in $F'$;
	\item the type-preserving subgroup of $G(F,F'')$ is simple.
\end{enumerate}
Then $N(F,F')$ is a simple subgroup of index eight in $G(F,F')$.
\end{thm}

\begin{proof}
First note that the second assumption implies by Proposition \ref{prop-g(f)simple} that $F$ is transitive and $F''$ is generated by its point stabilizers. Therefore $F'$ is also generated by its point stabilizers, so that $G(F,F')^+$ has index two in $G(F,F')$ (again by Proposition \ref{prop-g(f)simple}).  

Write $N = G(F,F') \cap G_0(F'',F') \cap G_1(F'',F')$. The two subgroups $G(F,F') \cap G_0(F'',F')$ and $G(F,F') \cap G_1(F'',F')$ are not equal, and have index two in $G(F,F')^{\star}$. So their intersection has index four in $G(F,F')^{\star}$, and consequently is of index exactly eight in $G(F,F')$. Moreover according to Corollary \ref{cor-monolithe}, the subgroup $N$ must contain $N(F,F')$, and to prove that $N$ is simple, it is enough to prove the converse inclusion.

Remark that $N(F,F')$ intersects $G(F,F'')^{\star}$ along a non-trivial normal subgroup of $G(F,F'')^{\star}$. Since $G(F,F'')^{\star}$ is simple by assumption, it follows that $N(F,F')$ must actually contain $G(F,F'')^{\star}$.

We let $g \in N$, and we prove that $g \in N(F,F')$ by induction on $|\Sigma(g)|$. If $\Sigma(g)$ is empty, then $g$ actually belongs to $G(F,F'')^{\star}$, and therefore $g \in N(F,F')$. When $\Sigma(g)$ is not empty, it must have even cardinality since $g \in G_0(F'',F') \cap G_1(F'',F')$, and moreover we may find $x_1 \neq x_2 \in \Sigma(g)$ at even distance from each other. Given such vertices, we apply Lemma \ref{lem-deux-sing} to $v = g(x_1)$ and $w = g(x_2)$. This provides us with an element $\gamma \in N(F,F')$ such that $\Sigma(\gamma) = \left\{g(x_1),g(x_2)\right\}$. Now consider the element $g' = \gamma g$. According to Lemma \ref{lem-sing-ind2}, we have \[\Sigma(g') = \Sigma(g) \triangle g^{-1} \Sigma(\gamma) = \Sigma(g) \setminus \left\{x_1, x_2\right\}. \] Consequently we can apply the induction hypothesis to $g'$, and we obtain that $g'$ belongs to $N(F,F')$. But $g = \gamma^{-1} g'$ and $\gamma \in N(F,F')$, so we deduce that $g$ belongs to $N(F,F')$ as well, and the proof is complete.
\end{proof}

Note that $F=F''$ is allowed in Theorem \ref{thm-simple-ind8}, in which case the second assumption becomes that $F$ is transitive and generated by its points stabilizers. 

\begin{cor} \label{cor-simpl-ind8}
Let $F \leq F'$ be permutation groups such that:
\begin{enumerate} [label=(\alph*)]
	\item $F$ has index two in $F'$;
	\item $F$ is transitive and generated by its points stabilizers.
\end{enumerate}
Then $G_0(F,F') \cap G_1(F,F')$ is a simple subgroup of index eight in $G(F,F')$.
\end{cor}

Examples of permutation groups satisfying these assumptions are $F = \mathrm{Alt}(d)$ and $F' = \mathrm{Sym}(d)$ for $d \geq 4$. For $d=4$ this is the only example. For $d=5$ we can take for $F$ the dihedral group $D_5$ and $F' = \mathbb{F}_5 \rtimes \mathbb{F}_5^\times$. A generalization of this example will be detailed at the end of Section \ref{sec-lattices}.

\section{Further properties of the groups $G(F,F')$} \label{sec-further-prop}

\subsection{Asymptotic dimension} \label{subsec-asdim}

Let $X$ be a metric space. Recall that $A,B \subset X$ are \textit{$r$-disjoint} if $d(a,b) \geq r$ for every $a \in A, b \in B$. Recall also that collection of subsets $(A_i)$ is \textit{uniformly bounded} if there is $C > 0$ so that $\mathrm{diam}(A_i) \leq C$ for all $i$. We say that $X$ has asymptotic dimension at most $n \geq 0$ if for every (large) $r > 0$, one can find $n+1$ uniformly bounded families $X_0,\ldots,X_n$ of $r$-disjoint sets, whose union is a cover of the space $X$. The \textit{asymptotic dimension} of $X$ is the smallest integer $n$ such that $X$ has asymptotic dimension at most $n$. Asymptotic dimension is an invariant of metric coarse equivalence, so that if $G$ is a locally compact compactly generated group, the asymptotic dimension of $G$ is well defined. 

\begin{prop}
Let $G$ be a locally compact compactly generated group acting on a locally finite tree $X$ such that all vertex stabilizers in $G$ are locally elliptic open subgroups. Then $G$ has asymptotic dimension at most one.
\end{prop}

\begin{proof}
Let $x_0$ be a vertex of $X$, and let $H = G_{x_0}$. Since the tree $X$ is locally finite, for every $r>0$, the coarse stabilizer $W_r(x_0) = \left\{g \in G \, : \, d(gx_0,x_0) \leq r \right\}$ of $x_0$ is a finite union of left cosets of $H$. Since the subgroup $H$ is locally elliptic, it has asymptotic dimension zero \cite[Proposition 4.D.4]{Cor-dlH}, and therefore by the previous observation $W_r(x_0)$ (endowed with the induced topology) has asymptotic dimension zero as well. So we are in position to apply Theorem 2 from \cite{BD-oadg}, which implies that $G$ has asymptotic dimension at most one. Note that the result is stated there for discrete groups, but the same proof works in the locally compact setting.
\end{proof}

This result applies notably to the family of groups $G(F,F')$, which clearly do not have asymptotic dimension zero.

\begin{cor} \label{cor-asdimone}
The group $G(F,F')$ has asymptotic dimension one.
\end{cor}

\subsection{Compact presentability}

Recall that the group $U(F)$ acts properly and cocompactly on $\treed$. So in particular $U(F)$ is coarsely simply connected, and therefore compactly presented \cite[Proposition 8.A.3]{Cor-dlH}. In this paragraph we characterize subgroups of $G(F,F')$ that are compactly presented, and show in particular that the groups $G(F,F')$ are not compactly presented when $F$ is a proper subgroup of $F'$. 


\begin{lem} \label{lem-group-tree-compact}
Let $G$ be locally compact compactly generated unimodular group, admitting a proper and continuous action on a tree $X$. Then compact open subgroups of $G$ have uniformly bounded Haar measure. In particular $G$ does not have non-compact locally elliptic open subgroups.
\end{lem}

\begin{proof}
Upon replacing $X$ by a minimal $G$-invariant subtree, one may assume that $X$ is a locally finite tree on which $G$ acts with finitely many orbits of vertices (see for example the second part of the proof of Lemma 2.4 in \cite{germs}). Since the action is proper, vertex stabilizers are compact open, and by the previous remark there are only finitely many conjugacy classes of vertex stabilizers. Moreover $G$ is unimodular, so vertex stabilizers have a finite number of possible Haar measures. Since every compact open subgroup has a subgroup of index at most two that is contained in a vertex stabilizer, the first statement is proved. The second statement follows because any non-compact locally elliptic open subgroup would be a strictly increasing union of compact open subgroups, which cannot happen.
\end{proof}

\begin{prop} \label{prop-notcp}
Let $G$ be a closed unimodular subgroup of $G(F,F')$. If $G$ is compactly presented, then the action of $G$ on $\treed$ is proper.

In particular, the group $G(F,F')$ is never compactly presented as soon as $F$ is a proper subgroup of $F'$.
\end{prop}

\begin{proof}
Since $G(F,F')$ has asymptotic dimension one by Corollary \ref{cor-asdimone}, it follows that $G$ must have asymptotic dimension zero or one. If $G$ has asymptotic dimension zero, then $G$ is compact because $G$ is compactly generated. So we may assume that $G$ has asymptotic dimension one. Since the fundamental group of a Cayley graph of $G$ is generated by loops of bounded length because $G$ is compactly presented, the group $G$ must be quasi-isometric to a tree according to \cite[Theorem 1.1]{FW}. This implies that the group $G$ must act geometrically on some locally finite tree (see \cite[Theorem 4.A.1]{Cor-focal-classi} and references therein), and since $G$ is unimodular, it follows from Lemma \ref{lem-group-tree-compact} that every locally elliptic open subgroup of $G$ must be compact. In particular vertex stabilizers in $G$ for its action on $\treed$ are compact, so the first statement is proved. 

The second statement follows from the first together with Corollary \ref{cor-u=g}.
\end{proof}

Proposition \ref{prop-notcp} applies notably to discrete subgroups of $G(F,F')$, and implies that any finitely presented discrete subgroup of $G(F,F')$ must intersect $G(F,F')_v$ along a finite subgroup, where $v$ is any vertex of $\treed$. In particular if $F$ is a proper subgroup of $F'$ and if $\Gamma$ is a lattice in $G(F,F')$, then $\Gamma$ cannot be finitely presented, because $\Gamma_v$ would be at the same time a finite group and a lattice in the non-compact group $G(F,F')_v$, which is impossible. 

\subsection{Relative abstract commensurators}

In this paragraph we give a second interpretation of the groups $G(F,F')$ in terms of relative commensurators (see Proposition \ref{prop-rel-comm}).

Let G be a profinite group, and $L$ an abstract group containing $G$. A \textit{relative commensurator} of $G$ in $L$ is an element of $L$ whose conjugation induces an isomorphism between two compact open subgroups of $G$. The set $\mathrm{Comm}_L(G)$ of relative commensurators of $G$ in $L$ is a group, which only depends on the local structure of $G$ in the sense that $\mathrm{Comm}_L(G) = \mathrm{Comm}_L(K)$ for any compact open subgroup $K$ of $G$. 

The idea of studying commensurators of profinite groups was initiated in \cite{BEW}. The motivation comes from the desire to study the structure of totally disconnected locally compact groups, by asking how much information we can recover about the ambient group by studying its local structure. This approach has been further investigated in \cite{germs} and in \cite{CRW-1,CRW-2}.

Here we investigate the relative commensurator of a compact open subgroup $K$ of $U(F)$ in $U(F')$. First remark that it follows from Lemma \ref{lem-topo-g(f)} that the group $G(F,F')$ commensurates $K$, so that we always have an inclusion $G(F,F') \leq \mathrm{Comm}_{U(F')}(K)$. The following result shows that $G(F,F')$ and $\mathrm{Comm}_{U(F')}(K)$ do not coincide in full generality.

\begin{prop} \label{prop-comm-normaliz}
If the group of relative commensurators of a compact open subgroup of $U(F)$ in $U(F')$ is equal to $G(F,F')$, then the normalizer of $F^+$ in $F'$ must be equal to $F$.
\end{prop}

\begin{proof}
We let $\sigma \in F'$ normalizing $F^+$, and we prove that $\sigma \in F$. We let $g \in U(F')$ such that $\sg = \sigma$ for every $v \in \verti$. Let $K$ be the pointwise stabilizer of an edge in $U(F)$, and let $h \in K$. Then for every vertex $v$, it readily follows from the multiplication rules (\ref{eq-rules}) that $\sigma(ghg^{-1},v) = \sigma \sigma(h,g^{-1}v) \sigma^{-1}$. Moreover since $h$ belongs to $U(F)$ and fixes an edge of $\treed$, we easily check that all the permutations $\sigma(h,w)$ belong to $F^+$. Now by definition $\sigma$ normalizes $F^+$, so we deduce that $ghg^{-1}$ actually belongs to $U(F)$. Therefore there is an open subgroup $K' \leq K$ such that $gK'g^{-1} \leq K$, which means that $g$ is a relative commensurator of $K$ in $U(F')$. By assumption $\mathrm{Comm}_{U(F')}(K) = G(F,F')$, so we deduce that there are only finitely many vertices $v$ such that $\sg \notin F$. This clearly implies that $\sigma$ belongs to $F$, and the proof is complete.
\end{proof}

\begin{rmq}
We point out that the conclusion of Proposition \ref{prop-comm-normaliz} implies in particular that $F$ must be equal to its normalizer in $F'$, but these two conditions are not equivalent, as the example $d=6$, $F=C_6$ and $F' = C_2 \wr C_3$ shows (where $C_n$ is the cyclic group of order $n$). 
\end{rmq}

Nevertheless, we prove in the following proposition that $G(F,F')$ does coincide with the group of relative commensurators of a compact open subgroup of $U(F)$ in $U(F')$ under the assumption that every point stabilizer $F_a$ is equal to its normalizer in $F_a'$. Note that this assumption covers many interesting cases. 


\begin{prop} \label{prop-rel-comm}
Assume that for every $a \in \Omega$, the group $F_a$ is equal to its normalizer in $F_a'$. Then $G(F,F')$ is equal to the group of relative commensurators of any compact open subgroup of $U(F)$ in $U(F')$.
\end{prop}

\begin{proof}
We let $g \in U(F')$ commensurating a compact open subgroup of $U(F)$, and we prove that $g$ has finitely many singularities. By assumption there exists a finite subtree $T$ of $\treed$ such that if we denote $U(F)_T=U_T$, then $gU_Tg^{-1} \leq U(F)$. We fix a vertex $v$ such that $g^{-1}(v) \notin T$, and we prove that $\sigma = \sigma(g,g^{-1}v) \in F$. Since all but finitely many vertices satisfy the condition $g^{-1}(v) \notin T$, this will prove the result. 

We let $a \in \Omega$ be the color of the unique edge emanating from $g^{-1}(v)$ and pointing toward the subtree $T$. We also let $\rho$ be an element of $F_a$, and we denote by $h$ an element of $U_T$ fixing $g^{-1}(v)$ and such that $\sigma(h,g^{-1}v) = \rho$. It follows from (\ref{eq-rules}) that $\sigma(ghg^{-1},v) = \sigma \rho \sigma^{-1}$. Now since $h \in U_T$, the element $ghg^{-1}$ remains in $U(F)$ by definition of $T$. According to the previous computation, this means that $\sigma \rho \sigma^{-1} \in F$, and we have proved that $\sigma F_a \sigma^{-1} \leq F$. Now since $F' \leq \hat{F}$, there exists a permutation $\tau \in F$ such that $\sigma \tau \in F_a'$, and we easily deduce that $\sigma \tau$ must lie in the normalizer of $F_a$ in $F_a'$. By assumption this latter group is reduced to $F_a$, so $\sigma \tau \in F$ and finally $\sigma \in F$.
\end{proof}

\section{Commensurating actions} \label{sec-ccc0}

\subsection{Diagrams} \label{sec-diag-g(f)}

In this paragraph we explain how the group $G(F,F')$ can be profitably studied by using a notion of diagrams introduced below. In the case when $F$ is transitive, one shows that this combinatorial data yields an estimate of the word-metric in the group $G(F,F')$ (see Proposition \ref{prop-N-qi-metric}). We will use this result later to prove that any closed inclusion $G(F,F')$ in  $G(H,H')$ is a quasi-isometric embedding (see Proposition \ref{prop-g(f,f')-qi}).

\bigskip

Recall that we have fixed an edge $e_0 \in \edg$ whose vertices are denoted $v_0$ and $v_1$. We identify $\Omega$ with the set of positive integers which are at most $d$, and we assume that $c(e_0)=1$. To every vertex $v \in \verti$, we associate the subtree of $\treed$ consisting of vertices whose projection to the geodesic between $v$ and $e_0$ is the vertex $v$. This subtree is naturally isomorphic to an infinite regular rooted tree, and will be denoted $\elle(v)$.


Given $g \in G(F,F')$, it readily follows from the fact that $g$ has only finitely many singularities that there exists a unique finite complete subtree $\arbmoins$ of $\treed$ such that:

\begin{enumerate}[label=(\roman*)]
\item $\arbmoins$ contains the edges $e_0$ and $g^{-1}(e_0)$;
\item for every vertex $v$ that is not an internal vertex of $\arbmoins$, we have $\sg \in F$;
\end{enumerate}
and being minimal for this property. We let $\arbplus$ be the image of $\arbmoins$ by $g$, and denote by $\Ne(g)$ the number of internal vertices of $\arbmoins$. Note that $\Ne(g)$ is also the number of internal vertices of $\arbplus$. We easily check that $\arbmoins = e_0$, or equivalently $\Ne(g) = 0$, if and only if $g$ belongs to $U(F)$ and stabilizes $e_0$.

\bigskip

Recall that a \textit{length function} on a group $\Gamma$ is a map $\mathcal{L}: \Gamma \rightarrow \mathbb{R}_+$ satisfying $\mathcal{L}(1) = 0$, $\mathcal{L}(g^{-1}) = \mathcal{L}(g)$ and $\mathcal{L}(gh) \leq \mathcal{L}(g) + \mathcal{L}(h)$ for every $g,h \in \Gamma$. 

\begin{lem} \label{lem-N-length}
The map $\Ne : G(F,F') \rightarrow \mathbb{R}_+$ is a length function on $G(F,F')$.
\end{lem}

\begin{proof}
By definition we have $\Ne(1) = 0$ and $\Ne(g) = \Ne(g^{-1})$ for every $g \in G(F,F')$. We let $g,h \in G(F,F')$, and we prove that $\Ne(gh) \leq \Ne(g) + \Ne(h)$. We denote $T = \mathcal{T}_h^- \cup h^{-1}(\mathcal{T}_g^-)$, which is a (complete) subtree because $\mathcal{T}_h^-$ and $h^{-1}(\mathcal{T}_g^-)$ both contain $h^{-1}(e_0)$. By construction $T$ contains the edges $e_0$ and $(gh)^{-1}(e_0)$, and $gh$ acts locally like $F$ outside $T$. By minimality it follows that $\mathcal{T}_{gh}^-$ must be a subtree of $T$, and in particular the number of internal vertices of $\mathcal{T}_{gh}^-$ is smaller than the number of internal vertices of $T$. The latter is at most $\Ne(g) + \Ne(h)$ by construction, so we have $\Ne(gh) \leq \Ne(g) + \Ne(h)$.
\end{proof}

So in particular the map $\Ne : G(F,F') \rightarrow \mathbb{R}_+$ gives rise to a left-invariant pseudo-metric on $G(F,F')$ defined by $\mathrm{dist}(g,h) = \Ne(g^{-1}h)$, and the aim of the rest of this subsection is to prove that when $F$ is transitive, this pseudo-metric is quasi-isometric to the word metric in $G(F,F')$. 

\begin{lem} \label{lem-decomp-stabv1}
For every $g \in G(F,F')$, there exist $\gamma \in U(F)$ and $g' \in G(F,F')_{v_1}$ such that $g = \gamma g'$ and $\Ne(g') \leq \Ne(g) + 1$.
\end{lem}

\begin{proof}
Let $g \in G(F,F')$. Since the group $U(F)$ is transitive on the set of vertices of $\treed$, we can choose some $\gamma \in U(F)$ such that $\gamma(v_1) = g(v_1)$, and set $g ' = \gamma^{-1} g$. Let us consider the complete subtree $T^-$ of $\treed$ obtained by adjoining if necessary the star around the vertex $v_1$ to the subtree $\arbmoins$. We check that $T^-$ contains the edges $e_0$ and $g'^{-1}(e_0)$, and that $g'$ acts locally like $F$ around every vertex of $\treed$ which is not an internal vertex of $T^-$. It follows that $\Ne(g')$ is at most equal to the number of internal vertices of $T^-$, which by construction is at most $\Ne(g) +1$. 
\end{proof}


Assume that the permutation group $F$ is transitive. Given $i \in \Omega \setminus \left\{1\right\}$, we choose some $\sigma_i \in F$ such that $\sigma_i(1)=i$, and we denote by $v_0^i$ (resp.\ $v_1^i$) the vertex of $\treed$ connected to $v_0$ (resp.\ $v_1$) by an edge having color $i$. Let us consider the bi-infinite line $\ell_i$ in $\treed$ defined by saying that $\ell_i$ contains the edge $e_0$, and the edge of $\ell_i$ in $\elle(v_0)$ (resp.\ $\elle(v_1)$) at distance $n$ from $e_0$ has color $\sigma_i^{-n}(1)$ (resp.\ $\sigma_i^{n}(1)$). We let $h_i$ be the hyperbolic isometry of $\treed$ of translation length one, whose axis is $\ell_i$, and such that $\sigma(h_i,v) = \sigma_i$ for every $v \in \verti$. Note that $h_i$ belongs to $U(F)$ and sends the subtree $\elle(v_1)$ onto the subtree $\elle(v_1^i)$, and $h_i^{-1}$ sends $\elle(v_0)$ onto $\elle(v_0^i)$.

Recall that for $v \in \verti$, we denote by $K_{0,F'}(v)$ the compact open subgroup of $G(F,F')$ consisting of elements fixing $v$ and not having any singularity outside the vertex $v$. Here for simplicity we write $K(v_0) = K_{0,F'}(v_0)$ and $K(v_1) = K_{0,F'}(v_1)$. In the sequel we let $S_H = \left\{h_2, \ldots, h_d\right\}$ and $S = S_H \cup K_{0,F'}(v_0) \cup K_{0,F'}(v_1)$, and the goal of the end of this paragraph is to prove that $S$ is a compact generating subset of $G(F,F')$ whose word length is comparable to $\Ne$.

The following result is the first technical lemma toward Proposition \ref{prop-N-qi-metric}.

\begin{lem} \label{lem-diag-word}
For every $g \in G(F,F')_{\elle(v_0)}$, we have $|g|_S \leq 3(d-1) \Ne(g) + 1$.
\end{lem}

\begin{proof}
Let us argue by induction on $\Ne(g)$. The case when $\Ne(g) = 0$ is easily settled, because $\Ne(g) = 0$ easily implies that $g \in U(F)_{v_1} \subset K(v_1)$, so we have $|g|_S \leq 1$.

Now assume that the result holds for every $g \in G(F,F')_{\elle(v_0)}$ with $\Ne(g) \leq n$ for some integer $n \geq 0$, and let $g \in G(F,F')_{\elle(v_0)}$ be such that $\Ne(g) = n+1$. We want to prove that the word length of $g$ is at most $3(d-1)(n+1)+1$. We may find $u \in K(v_1)$ such that $g' = ug$ fixes the star around the vertex $v_1$, and $u$ acts trivially on $\elle(v_0)$. The element $g'$ can therefore be written as a product $g' = g_2 \cdots g_d$, where $g_i \in G(F,F')$ acts trivially outside $\elle(v_1^{i})$. By construction we have $\Ne(g_i) \leq \Ne(g)$ and $\sum_i \Ne(g_i) \leq \Ne(g) + d-2$, because the vertex $v_1$ can be counted $d-1$ times in the sum, whereas it is counted only once in $\Ne(g)$. This last inequality can be rewritten as $\sum_i (\Ne(g_i)-1) \leq n$. Now for each $g_i$ different from the identity, let us consider the element $g_i' = h_i^{-1} g_i h_i$. Since $h_i$ sends the subtree $\elle(v_1)$ onto the subtree $\elle(v_1^i)$, the element $g_i'$ belongs to $G(F,F')_{\elle(v_0)}$. Moreover since all the local permutations of $h_i$ are equal to the same element of $F$, the set of singularities of $g_i'$ satisfies $S(g_i') \subset h_i^{-1}(S(g_i))$, and therefore $\Ne(g_i') \leq \Ne(g_i) - 1$. Therefore $\Ne(g_i')$ is at most $n$, so by the induction hypothesis the word length of $g_i'$ is at most $3(d-1)\Ne(g_i') + 1$. It follows that the word length of $g_i$ is at most $2 + 3(d-1)\Ne(g_i') + 1$, and we obtain \[ \begin{aligned} |g|_S & \leq 1 + \sum_{g_i \neq id} |g_i|_S \leq  1 + \sum_{g_i \neq id} \left(2 + 3(d-1)\Ne(g_i') + 1 \right) \\ & \leq  1 + 3(d-1) + 3(d-1) \sum_{g_i \neq id} \left(\Ne(g_i) - 1 \right) \\ & \leq  1 + 3(d-1) + 3(d-1)n = 3(d-1)(n+1)+1. \end{aligned} \]
\end{proof}

Note that the conclusion of Lemma \ref{lem-diag-word} also holds for elements of $G(F,F')_{\elle(v_1)}$, just by replacing $K(v_1)$ in the proof by $K(v_0)$, and each $h_i$ by its inverse. This allows us to obtain the following.

\begin{lem} \label{lem-length-gfv1}
For every $g \in G(F,F')_{v_1}$, we have $|g|_S \leq 3(d-1) \Ne(g) + 3$.
\end{lem}

\begin{proof}
Let $g \in G(F,F')$ fixing the vertex $v_1$. By definition of $K(v_1)$, there exists $u \in K(v_1)$ such that $g' = u g \in G(F,F')$ fixes the edge $e_0$. Note that since the element $u$ acts locally like $F$ at every vertex different from $v_1$, we have $\Ne(g') \leq \Ne(g)$. Now the element $g'$ can be written $g'=g_0'g_1'$, where $g_0' \in G(F,F')_{\elle(v_0)}$, $g_1' \in G(F,F')_{\elle(v_1)}$ satisfy $\Ne(g_0') + \Ne(g_1') = \Ne(g')$. Therefore Lemma \ref{lem-diag-word} can be applied to these elements, and we obtain \[ |g|_S \leq 1 + |g'_0|_S + |g'_1|_S \leq 1 + 3(d-1)\Ne(g_0') + 1 + 3(d-1)\Ne(g_1') + 1 \leq 3(d-1) \Ne(g) + 3. \]
\end{proof}

\begin{lem} \label{lem-length-uf}
For every $\gamma \in U(F)$, we have $|\gamma|_S \leq d(\gamma(v_1),v_1) + 1$.
\end{lem}

\begin{proof}
We argue by induction on $d(\gamma(v_1),v_1)$. If $\gamma$ fixes $v_1$ then $\gamma$ belongs to $K(v_1)$ and therefore $|\gamma|_S \leq 1$. Assume that $|\gamma|_S \leq d(\gamma(v_1),v_1) + 1$ for every $\gamma \in U(F)$ such that $d(\gamma(v_1),v_1) \leq n$, and let $\gamma \in U(F)$ be such that $d(\gamma(v_1),v_1) = n + 1$. If the vertex $\gamma(v_1)$ belongs to the subtree $\elle(v_1)$, then there exists some integer $i$ such that $\gamma' = h_i^{-1}\gamma \in U(F)$ satisfies $d(\gamma'(v_1),v_1) \leq n$. By the induction hypothesis, the word length of $\gamma'$ is at most $n+1$, and we deduce that $|\gamma|_S \leq n+2$. Now if $\gamma(v_1)$ belongs to $\elle(v_0)$ then the same argument can be applied to $h_i \gamma$ for some integer $i$.
\end{proof}

We are finally able to give the following precise estimate for the word metric in $G(F,F')$.

\begin{prop} \label{prop-N-qi-metric}
Assume that $F$ is transitive. Then $S$ is a compact generating subset of $G(F,F')$, and for every $g \in G(F,F')$, we have \[ \Ne(g) \leq |g|_S \leq (3d-2) \Ne(g) + 3d+2. \]
\end{prop}

\begin{proof}
The lower bound easily follows from the fact that the function $\Ne$ is subadditive by Lemma \ref{lem-N-length} and takes values $0$ or $1$ on elements of $S$. Let us prove the upper bound. According to Lemma \ref{lem-decomp-stabv1}, one can write $g = \gamma g'$ with $\gamma \in U(F)$ and $g' \in G(F,F')_{v_1}$ such that $\Ne(g') \leq \Ne(g) + 1$. It follows from Lemma \ref{lem-length-uf} that the word length of $\gamma$ satisfies \[ |\gamma|_S \leq d(\gamma(v_1),v_1) +1 = d(g(v_1),v_1) +1 \leq \Ne(g) +2. \] On the other hand, we can apply Lemma \ref{lem-length-gfv1} to the element $g'$, which yields \[ |g'|_S \leq 3(d-1) \Ne(g') + 3 \leq 3(d-1) (\Ne(g)+1) + 3. \] We finally obtain \[ |g|_S \leq |\gamma|_S + |g'|_S \leq \Ne(g) + 2 + 3(d-1) (\Ne(g)+1) + 3 = (3d-2) \Ne(g) + 3d+2. \]
\end{proof}

\subsection{A commensurating action of $G(F,F')$}

In this subsection we prove that the group $G(F,F')$ admits a commensurating action whose corresponding cardinal definite function is equal to twice the function $\Ne$ (see Proposition \ref{prop-commens-g(f,f')}). By a general argument, we obtain a proper action of $G(F,F')$ on a CAT(0) cube complex. This cube complex is infinite dimensional, and we actually prove that $G(F,F')$ cannot act properly on a finite dimensional CAT(0) cube complex.

\bigskip

Recall that $e_0$ is a fixed edge of $\treed$ having color $c(e_0) = 1$. Let $H$ denote the open subgroup of $G(F,F')$ consisting of elements $g$ stabilizing $\elle(v_0)$ setwise, and such that $\sigma(g,w) \in F$ for every vertex $w$ in $\elle(v_0)$. Equivalently, \[ H = \left\{ g \in G(F,F')_{e_0} : S(g) \subset \elle(v_1)\right\}. \] For every vertex $v \in \verti$, we let $M_v$ be the set of elements $g \in G(F,F')$ such that $g(\elle(v_0)) = \elle(v)$ and $\sigma(g,w) \in F$ for every vertex $w$ in $\elle(v_0)$.

\begin{lem} \label{lem-Mcoset}
For every $v \in \verti$, $M_v$ is either empty or equal to a single $H$-coset.
\end{lem}

\begin{proof}
Let us check that all the elements of $M_v$ belong to the same left coset of $H$. If $g_1, g_2 \in M_v$, then $g_1^{-1}g_2$ must stabilize setwise $\elle(v_0)$. Moreover for every vertex $w$ in $\elle(v_0)$, we have $\sigma(g_1^{-1}g_2,w) = \sigma(g_1,g_1^{-1}g_2 w)^{-1} \sigma(g_2,w) \in F$ because $w$ and $g_1^{-1}g_2 w$ are vertices of $\elle(v_0)$, so $g_1^{-1}g_2 \in H$. Thus, if non-empty, $M_v$ is contained in exactly one $H$-coset.

Conversely if $g \in M_v$ and $h \in H$, then $(gh)(\elle(v_0)) = g(\elle(v_0)) = \elle(v)$, and for every vertex $w$ in $\elle(v_0)$, we have $\sigma(gh,w) = \sigma(g,hw) \sigma(h,w)$. Now $\sigma(h,w) \in F$ because $h \in H$, and $ \sigma(g,hw) \in F$ because $h(w)$ remains in $\elle(v_0)$ and $g \in M_v$. So $\sigma(gh,w) \in F$, and we have proved that $gh \in M_v$.
\end{proof}

\begin{lem} \label{lem-mvempty}
For every $v \in \verti$, the set $M_v$ is empty if and only if the color of the unique edge around $v$ that is not in $\elle(v)$ is not in the $F$-orbit of $1$.
\end{lem}

\begin{proof}
Let $e_v$ be the unique edge around $v$ that is not in $\elle(v)$, whose color is denoted by $i_v$.

Assume that $M_v$ is non-empty, and let $g \in M_v$. Since $g$ sends the subtree $\elle(v_0)$ onto $\elle(v)$, we have $g(e_0) = e_v$. But the permutation $\sigma(g,v_0)$ preserves the orbits of $F$, so it follows that $i_v$ is in the $F$-orbit of $c(e_0) = 1$.

For the converse implication, let $\sigma \in F$ such that $\sigma(1) = i_v$. We let $\gamma$ be the automorphism of $\treed$ define by declaring that $\gamma(v_0) = v$, and $\sigma(\gamma,w) = \sigma$ for every vertex $w$. Clearly $\gamma \in U(F)$ and by construction $\gamma$ must send $\elle(v_0)$ onto $\elle(v)$ because $\sigma(\gamma,v_0)(1) = i_v$. So $\gamma \in M_v$, which is therefore non-empty.
\end{proof}

We will need the following lemma.

\begin{lem} \label{lem-int-vertices}
Given $g \in G(F,F')$ and $v \in \verti$, the following statements are equivalent:
\begin{enumerate}[label=(\roman*)]
\item $g(\elle(v)) = \elle(g(v))$ and $\sigma(g,w) \in F$ for every vertex $w$ in $\elle(v)$;
\item $v$ is not an internal vertex of $\arbmoins$.
\end{enumerate}
\end{lem}

\begin{proof}
By construction $g$ sends the complement of $\arbmoins$ onto the complement of $\arbplus$ locally like $F$, so it is clear that if $v$ is not an internal vertex of $\arbmoins$ then $g(\elle(v)) = \elle(g(v))$ and $\sigma(g,w) \in F$ for every vertex $w$ in $\elle(v)$. For the converse implication, remark that if $v$ is an internal vertex of $\arbmoins$, then either there is a vertex $w$ in $\elle(v)$ such that $\sigma(g,w) \notin F$, or the edge $g^{-1}(e_0)$ belongs to $\elle(v)$. This last property implies that $g(\elle(v))$ contains the edge $e_0$, and therefore cannot be equal to $\elle(g(v))$.
\end{proof}


Assume that $F$ is transitive. According to Lemma \ref{lem-mvempty}, this assumption ensures that the set $M_v$ is non-empty for every $v \in \verti$.

\begin{lem} \label{lem-Mvv'}
Let $v,v' \in \verti$ and $g \in G(F,F')$. Then $g M_v$ and $M_{v'}$ are either disjoint or equal, and $g M_v = M_{v'}$ if and only if $v' = g(v)$ and $v$ is not an internal vertex of $\arbmoins$.
\end{lem}

\begin{proof}
The fact that $g M_v$ and $M_{v'}$ are either disjoint or equal follows immediately from Lemma \ref{lem-Mcoset}. Assume that $g M_v = M_{v'}$. Since the subset $M_v$ is non-empty, there exists $g_v \in G(F,F')$ such that $g_v(\elle(v_0)) = \elle(v)$ and $\sigma(g_v,w) \in F$ for every vertex $w$ in $\elle(v_0)$. Since $gg_v \in M_{v'}$ by assumption, we have $gg_v(\elle(v_0)) = \elle(v')$ and $\sigma(gg_v,w) \in F$ for every vertex $w$ in $\elle(v_0)$. In particular we have $g(\elle(v)) = \elle(v')$, so $v' = g(v)$. Since $\sigma(gg_v,w) = \sigma(g,g_v w) \sigma(g_v,w)$ and $\sigma(gg_v,w), \sigma(g_v,w) \in F$, we obtain that $\sigma(g,w') \in F$ for every vertex $w'$ in $\elle(v)$. According to $(i) \Rightarrow (ii)$ of Lemma \ref{lem-int-vertices}, this implies that the vertex $v$ is not an internal vertex of $\arbmoins$.

Conversely assume that $v$ is not an internal vertex of $\arbmoins$. According to the implication $(ii) \Rightarrow (i)$ of Lemma \ref{lem-int-vertices}, we have $g(\elle(v)) = \elle(g(v))$ and $\sigma(g,w) \in F$ for every vertex $w$ in $\elle(v)$. By the same argument as above, it follows that $g M_v \subset M_{g(v)}$, and therefore $g M_v = M_{g(v)}$.
\end{proof}

We denote by $M \subset G(F,F')$ the union of the subsets $M_v$, when $v$ ranges over the set of vertices $\verti$. Since $M$ is a union of left cosets of $H$, we identity the subset $M$ of $G(F,F')$ with its image in $G(F,F') / H$.

Recall that if $G$ is a group acting on a set $X$, a subset $A \subset X$ is \textit{commensurated} by $G$, or $G$ \textit{commensurates} $A$, if $\#(gA \triangle A)$ is finite for every $g \in G$. 

\begin{prop} \label{prop-commens-g(f,f')}
Assume that $F$ is transitive. Then the action of $G(F,F')$ on $G(F,F') / H$ commensurates the subset $M$. More precisely, we have $\#(gM \triangle M) = 2 \Ne(g)$ for every $g \in G(F,F')$. 
\end{prop}

\begin{proof}
Let $g \in G(F,F')$. According to Lemma \ref{lem-Mvv'}, the subset $gM \backslash M$ is the union of $g M_v$, where $v$ ranges over the set of internal vertices of $\arbmoins$. Since none of these $g M_v$ is empty, this union consists exactly in $\Ne(g)$ left cosets of $H$, and therefore $\#(gM \backslash M) = \Ne(g)$. By applying the same argument to $g^{-1}$, we obtain \[ \#(gM \triangle M) = \#(gM \backslash M) + \#(M \backslash gM) = \Ne(g) + \Ne(g^{-1}) = 2\Ne(g). \]
\end{proof}

By a general principle (see for instance \cite[Proposition 5.17]{Cor-PW} and references therein), we deduce the following result.

\begin{cor} \label{cor-ccc0}
Assume that $F$ is transitive. Then there exist a CAT(0) cube complex $C$ on which $G(F,F')$ acts properly, and a vertex $x_0 \in C$ such that in the $\ell^1$-metric, $d(gx_0,x_0) = 2 \Ne(g)$ for every $g \in G(F,F')$.
\end{cor}


Corollary \ref{cor-ccc0} reveals that Proposition \ref{prop-N-qi-metric} established in the previous subsection has a geometric interpretation: it exactly means that the orbital map $G(F,F') \rightarrow C$, $g \mapsto gx_0$, is a quasi-isometric embedding.

\begin{rmq} \label{rmq-not-cocompact}
The action of $G(F,F')$ on this CAT(0) cube complex is not cocompact when $F$ is a proper subgroup of $F'$, and more generally one cannot hope that $G(F,F')$ acts properly and cocompactly by isometries on a simply connected metric space. The reason is that the existence of such an action would imply that $G(F,F')$ is coarsely simply connected and therefore compactly presented \cite[Proposition 8.A.3]{Cor-dlH}, a contradiction with Proposition \ref{prop-notcp}.
\end{rmq}

The end of this section is devoted to the proof that, although $G(F,F')$ does act properly on a CAT(0) cube complex, it cannot act properly on a \textit{finite dimensional} CAT(0) cube complex (see Proposition \ref{prop-not-dim-finie}). The argument will consist in embedding in $G(F,F')$ a compact extension of the wreath product $C_{p} \wr \mathbb{F}_2$, and using the fact that the latter group does not admit such an action. 


\bigskip

If $H \leq H'$ and $G$ are groups, we call the \textit{semi-restricted wreath product} of $H,H'$ and $G$ the set of pairs $(f,g)$ where $g \in G$ and $f: G \rightarrow H'$ is such that $f(\gamma) \in H$ for all but finitely many $\gamma \in G$. It is a subgroup of the unrestricted wreath product of $H'$ and $G$, which will be denoted $(H,H') \wr G$. Note that $(H,H') \wr G$ always contains the restricted (or standard) wreath product $H' \wr G$.

For every $a \in \Omega$, we denote by $U^a$ (resp.\ $G^a$) the pointwise stabilizer in the group $U(F)$ (resp.\ $G(F,F')$) of a half-tree of $\treed$ defined by an edge $e \in \edg$ such that $c(e)=a$. 

\begin{prop} \label{prop-wr-g(f,f')}
For every $a \in \Omega$, the semi-restricted wreath product \mbox{$(U^a, G^a) \wr \mathbb{F}$} embeds as a subgroup of $G(F,F')$, where $\mathbb{F}$ is a free group of rank $d-2$.
\end{prop}

\begin{proof}
Let us consider the largest subtree $T$ of $\treed$ containing the vertex $v_0$ and such that all the edges $e$ of $T$ satisfy $c(e) \neq a$. Note that $T$ is a regular tree of degree $d-1$. We let $\mathbb{F}$ be the subgroup of $U(\left\{1\right\})^{\star}$ stabilizing $T$. The group $\mathbb{F}$ acts freely and without inversion on $T$, and the quotient $\mathbb{F} \backslash T$ has two vertices and $d-1$ non-oriented edges, so it follows that $\mathbb{F}$ is free of rank $(d-1) - 2 + 1 = d-2$ \cite[Theorem 4']{Serre-trees}. 

Let us denote by $V_0$ the set of vertices of $T$ at even distance from $v_0$. For every vertex $v \in V_0$, let $e_v$ be the edge of $\treed$ containing $v$ and such that $c(e_v)=a$, and we denote by $T^v$ the unique half-tree defined by $e_v$ not containing $T$. We also denote by $\Gamma_a^v$ the subgroup of $G(F,F')$ fixing $T$, acting on $T^w$ by an element of $G^a$ if $w=v$, and by an element of $U^a$ otherwise; and being the identity elsewhere. By construction the subgroup of $G(F,F')$ generated by all the $\Gamma_a^v$, $v \in V_0$, is the subgroup $\prod_v^0 G^a$ of $\prod_v G^a$ consisting of elements having all but finitely many of their coordinates in $U^a$. Now the group $\mathbb{F}$ permutes the subtrees $T^v$, so it follows that the subgroup $\Gamma$ of $G(F,F')$ generated by $\mathbb{F}$ together with all the $\Gamma_a^v$ is isomorphic to $\prod_v^0 G^a \rtimes \mathbb{F}$. Moreover since $\mathbb{F}$ acts simply transitively on $V_0$, we deduce that $\Gamma$ is exactly the semi-restricted wreath product $(U^a,G^a) \wr \mathbb{F}$. 
\end{proof}

\begin{prop} \label{prop-not-dim-finie}
For every $d \geq 4$ and every permutation groups $F \lneq F'$, the group $G(F,F')$ cannot act properly on a finite dimensional CAT(0) cube complex. 
\end{prop}

\begin{proof}
We first claim that for every $a \in \Omega$, there exist an integer $k \geq 1$ and a prime $p$ such that $F_a'$ contains $C_{p^k}$ and $C_{p^k} \cap F = C_{p^{k-1}}$. Indeed, since $F$ is a proper subgroup of $F'$, one can find an element $x$ in $F_a' \setminus F$. Without loss of generality we may assume that the order of $x$ is a prime power, and the claim follows by considering the subgroup generated by some suitable power of $x$.

By combining this observation with Proposition \ref{prop-wr-g(f,f')}, we deduce that the semi-restricted wreath product $(C_{p^{k-1}}, C_{p^{k}}) \wr \mathbb{F}$ embeds as a (closed) subgroup in $G(F,F')$. Since $\mathrm{rk}(\mathbb{F}) = d-2 \geq 2$, one can find in $\mathbb{F}$ a non-abelian free group of rank two $\mathbb{F}_2$, and therefore the group $(C_{p^{k-1}}, C_{p^{k}}) \wr \mathbb{F}_2$ also embeds as a closed subgroup in $G(F,F')$. 

Assume that $G(F,F')$ has a proper action on a finite dimensional CAT(0) cube complex. Then according to the previous paragraph, the group $H = (C_{p^{k-1}}, C_{p^{k}}) \wr \mathbb{F}_2$ has the same property. Since the normal subgroup $K = \prod_{\mathbb{F}_2} C_{p^{k-1}}$ is compact, $K$ must have fixed points \cite[Corollary II.2.8]{BH}. So we would obtain a proper action of $H / K \simeq C_{p} \wr \mathbb{F}_2$ on the set of fixed points of $K$, which (upon passing to the barycentric subdivision) is again a finite dimensional CAT(0) cube complex. Now as observed in \cite{CSVa}, it follows from \cite[Corollary 2.12]{Oz-P} together with \cite{Guen-Hig} that the restricted wreath product $C_{p} \wr \mathbb{F}_2$ cannot act properly on a finite dimensional CAT(0) cube complex. So we have reached a contradiction, and the proof is complete.
\end{proof}

\begin{rmq}
The above proof actually shows that when $d \geq 4$ and $F \lneq F'$, the standard wreath product $C_{p} \wr \mathbb{F}_2$ embeds as a \textit{discrete} subgroup of $G(F,F')$ as soon as there exists an element of prime order $p$ in $F' \setminus F$ fixing a point of $\Omega$.
\end{rmq}

\section{Lattices} \label{sec-lattices}

\subsection{Embeddings}

This paragraph concerns the study of the properties of inclusions of the groups $G(F,F')$ into each other for a fixed $d \geq 3$.

\begin{center}
\textit{For all this subsection we fix some permutation groups $F \leq F'\leq \hat{F}$ and $H \leq H' \leq \hat{H}$ such that $F \leq H$ and $F' \leq H'$.}
\end{center}

These conditions imply that $G(F,F')$ is a subgroup of $G(H,H')$. The following lemma is easy, and we leave the proof to the reader.

\begin{lem}
The inclusion $G(F,F') \hookrightarrow G(H,H')$ is:
\begin{enumerate} [label=(\alph*)]
  \item open if and only if $H_a \leq F$ for every $a \in \Omega$;
	\item closed if and only if $H \cap F' = F$;
	\item discrete if and only if $H \cap F'$ acts freely on $\Omega$.
\end{enumerate}
\end{lem}

We derive from Proposition \ref{prop-N-qi-metric} the following interesting result, which says that every closed inclusion between the groups $G(F,F')$ is undistorted.

\begin{prop} \label{prop-g(f,f')-qi}
Suppose that $F$ is transitive and $H \cap F' = F$. Then the group $G(F,F')$ is quasi-isometrically embedded inside $G(H,H')$.
\end{prop}

\begin{proof}
Let $g \in G(F,F')$, and $v \in \verti$. If $\sg$ does not belong to $F$ then it does not belong to $H$ either, in view of the fact that $F'$ and $H$ intersect along $F$. This means that $g$ has the same set of singularities when viewed as an element of $G(F,F')$ and $G(H,H')$. Therefore we have $\mathcal{N}_{F,F'}(g) = \mathcal{N}_{H,H'}(g)$, and the conclusion then follows from Proposition \ref{prop-N-qi-metric}.
\end{proof}

\begin{prop} \label{prop-incl-cocomp}
Suppose that $H' = H F'$. Then the group $G(F,F')$ has cocompact closure in $G(H,H')$.
\end{prop}

\begin{proof}
We shall prove that $G(H,H') = K \cdot G(F,F')$, where $K = U(H)_{v_0}$. Clearly it is enough to prove $G(H,H')_{v_0} = K \cdot G(F,F')_{v_0}$. We let $g \in G(H,H')_{v_0}$, and we argue by induction on the cardinality of $S(g)$. If $S(g)$ is empty, then $g \in K$ and the result is trivial. Assume that $g$ has $n+1$ singularities, $n \geq 0$, and let $v \in S(g)$. 

We deal with the case $v \neq v_0$ (the case $v=v_0$ being similar). Let $e$ be the edge emanating from $v$ and pointing toward $v_0$, and let $a=c(e)$. Since $F' \leq \hat{F}$, one can check that the assumption $H' = H F'$ implies $H' = H F_a'$, and it follows that there exists $\sigma \in F_a'$ such that $\sg \sigma \in H$. Therefore if we take $\gamma \in G(F,F')_{v_0}$ fixing the half-tree emanating from $e$ and containing $v_0$, and such that $\sigma(\gamma,v) = \sigma$ and $v$ is the only singularity of $\gamma$, then $g' = g \gamma \in G(H,H')_{v_0}$ has at most $n$ singularities. By induction there is $\gamma' \in G(F,F')_{v_0}$ such that $g' \gamma' \in K$, and therefore $g (\gamma \gamma') \in K$.
\end{proof}

We highlight the following consequence for future reference.

\begin{cor} \label{cor-coco-lattice}
Suppose that $H \cap F' = F$ and $H' = H F'$. Then $G(F,F')$ is a closed cocompact subgroup of $G(H,H')$.

If moreover $F$ acts freely on $\Omega$, then it is a cocompact lattice.
\end{cor}

\subsection{Iterated wreath products and lattices}

In this paragraph we study the existence of lattices in a family of locally compact groups which appear as union of infinitely iterated permutational wreath products. We give a very short proof that some of the groups under consideration do not have lattices (see Theorem \ref{thm-L-no-lattice}), and apply this result to the groups $G(F,F')$ (see Corollary \ref{cor-without-lattice}).

\bigskip

We let $X$ be a finite set of cardinality $\ell \geq 2$, and we fix some permutation groups $D \leq D' \leq \Sx$. To avoid confusion, we intentionally do not use the notation $F$ and $F'$ for the permutation groups, because the present construction will actually be applied to point stabilizers in $F$ and $F'$. 


We let $W_0(D) = 1$ and $W_{n+1}(D) = D \wr W_{n}(D)$ for $n \geq 0$, where wreath products are considered with their imprimitive wreath product action. The group $W_{n}(D')$ is defined similarly. We denote by $L_0$ the infinitely iterated wreath product \[L_0 = \ldots \wr D \wr \ldots \wr D,\] which is the projective limit of the finite groups $W_{n}(D)$, and denote by $U_n$ the kernel of the natural projection of $L_0$ onto $W_{n}(D)$. 

For $n \geq 0$ we also let \[L_n = \ldots \wr D \wr \ldots \wr D \wr D' \wr \ldots \wr D', \] where the permutation group $D'$ appears $n$ times. Each $L_n$ is a subgroup of the infinitely iterated wreath product of $\mathrm{Sym}(X)$, and since $D$ is a subgroup of $D'$, $L_n$ is a subgroup of $L_{n+1}$ for every $n \geq 0$. We denote by $L(D,D')$ the increasing union of the groups $L_n$. Endowed with the topology making the inclusion of $L_0$ a continuous open map, $L(D,D')$ is a locally elliptic totally disconnected locally compact group.

\bigskip

The following lemma gives a general obstruction for a locally compact group to contain lattices. Recall that a subgroup $H \leq G$ is said to be \textit{essential} if $H$ intersects non-trivially every non-trivial subgroup of $G$.

\begin{lem} \label{lem-general-no-lattice}
Let $G$ be a locally compact group, and $\mu$ a Haar measure on $G$. Let $(U_n)$ be a basis of neighbourhoods of the identity consisting of compact open subgroups. Assume that there exists a sequence of subgroups $(K_n)$ such that:
\begin{enumerate}[label=(\alph*)]
\item $K_n$ contains $U_n$ as an essential subgroup;
\item $\mu(K_n) \rightarrow \infty$ when $n \rightarrow \infty$.
\end{enumerate}
Then for every $k \geq 1$, the group $G^k$ does not have lattices. 
\end{lem}

\begin{proof}
If $G$ satisfies these assumptions, then so does $G^k$ for $k \geq 1$, so it is enough to give the proof for $k=1$. Assume that $\Gamma$ is a lattice in $G$. Since the Haar measure of $K_n$ goes to infinity, $\Gamma$ must intersect $K_n$ non-trivially for $n$ large enough, and therefore $\Gamma$ must intersect $U_n$ non-trivially as well thanks to the first assumption. This means that $\Gamma$ intersects non-trivially any neighbourhood of the identity, and therefore $\Gamma$ cannot be discrete. Contradiction. 
\end{proof}

The following is the main result of this paragraph. 

\begin{thm} \label{thm-L-no-lattice}
Let $D \leq D' \leq \Sx$, and $\ell = |X|$. Assume that:
\begin{enumerate} [label=(\alph*)]
  \item $D$ is an essential subgroup of $D'$;
	\item $|D'| < \left( D': D \right)^{\ell}$.
\end{enumerate}
Then for every $k \geq 1$, the group $L(D,D')^k$ does not have lattices.
\end{thm}

\begin{proof}
Let us fix a Haar measure $\mu$ on $L(D,D')$, normalized so that $\mu(L_0) = 1$. For $n \geq 0$, we consider the kernel $K_n = U_{n+1} \rtimes (D')^{\ell^{n}}$ of the natural projection of $L_{n+1}$ on $W_{n}(D')$, and we shall prove that the sequence $(K_n)$ satisfies the assumptions of Lemma \ref{lem-general-no-lattice}.

Since being an essential subgroup is stable by taking finite direct products, $D^{\ell^{n}}$ is an essential subgroup of $(D')^{\ell^{n}}$, and we deduce that $U_{n} = U_{n+1} \rtimes D^{\ell^{n}}$ is essential in $K_n$. 

Now the Haar measure of $K_n$ is equal to $\mu(K_n) = \mu(U_{n+1}) |D'|^{\ell^{n}}$, and since $U_0 = L_0$ has measure one, we have \[ \mu(U_{n+1}) = (U_{0}:U_{n+1})^{-1} = |W_{n+1}(D)|^{-1} = |D|^{-\frac{\ell^{n+1}-1}{\ell-1}}, \] where the last equality is easily obtained by induction. Therefore \[ \mu(K_n) = \frac{|D'|^{\ell^{n}}}{|D|^{\frac{\ell^{n+1}-1}{\ell-1}}} = |D|^{1 / (\ell-1)} \left( \frac{|D'|}{|D|^{\ell / (\ell-1)}} \right)^{\ell^{n}}.\] Now the assumption $|D'| < \left( D': D \right)^{\ell}$ is easily seen to be equivalent to $|D'| > |D|^{\ell / (\ell-1)}$. Therefore the above computation shows that the Haar measure of $K_n$ goes to infinity, and the conclusion then follows from Lemma \ref{lem-general-no-lattice}. 
\end{proof}

Note that when $G$ is a finite group, a subgroup $H$ is essential if and only if $H$ contains all elements of prime order. This is for instance the case when there is a short exact sequence $1 \rightarrow H \rightarrow G \rightarrow C_p \rightarrow 1$ that does not split. Examples of permutation groups satisfying the assumptions of Theorem \ref{thm-L-no-lattice} are $D = C_k$ and $D' = C_\ell$, where $\ell=pk$ and $p$ is a prime dividing $k$, and $C_k$ and $C_\ell$ act by translation on $C_\ell$.

\bigskip

We derive the following consequence for the groups $G(F,F')$.

\begin{cor} \label{cor-without-lattice}
Let $F \leq F' \leq \Sy$ be permutation groups such that $F$ is transitive, and let $a \in \Omega$ and $d = |\Omega|$. Assume that there exists a subgroup $D'$ with $F_a \leq D' \leq F_a'$ and such that:
\begin{enumerate} [label=(\alph*)]
  \item $F_a$ is an essential subgroup of $D'$;
	\item $|D'| < \left( D': F_a \right) ^{d-1}$.
\end{enumerate}
Then the group $G(F,F')$ does not have lattices. 
\end{cor}

\begin{proof}
Since $F$ is assumed to be transitive, all point stabilizers in $F$ are conjugate, and the stabilizer of an edge in $G(F,F')$ is isomorphic to the group $L(F_a,F_a')^2$. Therefore if we write $D = F_a$, the group $G(F,F')$ contains $L(D,D')^2$ as an open subgroup. It follows that any lattice in $G(F,F')$ would intersect $L(D,D')^2$ along a lattice in $L(D,D')^2$, and the conclusion then follows from Theorem \ref{thm-L-no-lattice} applied with $k=2$.
\end{proof}

\begin{ex}
$G(F,F')$ does not have lattices when $F = \mathrm{PSL}(2,q)$ and $F' = \mathrm{PGL}(2,q)$ acting on the projective line $\mathbb{P}^1(\mathbb{F}_q)$, and $q=1 \! \! \mod 4$.
\end{ex}

Indeed, the stabilizer of the point $\infty$ in $F$ is $F_{\infty} = \mathbb{F}_q \rtimes \mathbb{F}_q^{\times,2}$, where $\mathbb{F}_q^{\times,2}$ is the set of non-zero squares in $\mathbb{F}_q$. Since $q=1 \! \! \mod 4$, the element $-1$ is a square in $\mathbb{F}_q$, and it follows that the short exact sequence \[ 1 \rightarrow \mathbb{F}_q^{\times,2} \rightarrow \mathbb{F}_q^{\times} \rightarrow C_2 \rightarrow 1 \] does not split, and that $F_{\infty}$ is essential in $F_{\infty}' = \mathbb{F}_q \rtimes \mathbb{F}_q^\times$. Moreover \[|F_{\infty}'| = q(q-1) < 2^q = \left( F_{\infty}': F_{\infty} \right) ^{|\Omega|-1},\] so Corollary \ref{cor-without-lattice} applies.

\subsection{Proofs of the results}

We finally prove the results stated in the introduction.

We claim that there are many permutation groups giving rise to groups $G(F,F')$ as in Theorem \ref{thm-simple-coc-no-lattice}, and we shall detail one family of examples. Note that this will prove Corollary \ref{cor-simple-no-lattice} at the same time in view of Corollary \ref{cor-asdimone}.

\begin{proof}[Proof of Theorem \ref{thm-simple-coc-no-lattice}]
Let $q=1 \! \! \mod 4$ be a prime power, $n \geq 1$ and $\Omega=\mathbb{F}_q^n$. Let $F'$ be the affine group $\mathbb{F}_q^n \rtimes \mathrm{GL}(n,q)$, and let $F = \mathbb{F}_q^n \rtimes \mathrm{GL}^2(n,q)$ be its subgroup of index two consisting of elements whose linear part has a determinant that is a square. 

We claim that $F$ lies in $\mathrm{Alt}(\Omega)$ and that $F'$ does not. Indeed, consider the decomposition $\mathrm{GL}(n,q) = \mathrm{SL}(n,q) \rtimes \mathbb{F}_q^{\times}$, where $\mathbb{F}_q^{\times}$ is embedded in $\mathrm{GL}(n,q)$ via $x \mapsto \mathrm{diag}(x,1,\ldots,1)$. If $x$ is a generator of $\mathbb{F}_q^{\times}$, we easily see that, viewed as a permutation on $\Omega=\mathbb{F}_q^n$, the element $x$ is a cycle of length $q-1$. This shows that $F'$ does not lie in $\mathrm{Alt}(\Omega)$. Now let us denote by $\mathbb{F}_q^{\times,2}$ the set of non-zero squares in $\mathbb{F}_q$, and check that $F = \mathbb{F}_q^n \rtimes (\mathrm{SL}(n,q) \rtimes \mathbb{F}_q^{\times,2})$ is a subgroup of $\mathrm{Alt}(\Omega)$, by verifying that the three subgroups $\mathbb{F}_q^n$, $\mathrm{SL}(n,q)$ and $\mathbb{F}_q^{\times,2}$ are all in $\mathrm{Alt}(\Omega)$. Every non-zero element of $\mathbb{F}_q^n$ is a product of $p$-cycles, where $p$ is the characteristic of $\mathbb{F}_q$, so these permutations are alternating. Since $\mathrm{SL}(n,q)$ is perfect, it must be a subgroup of $\mathrm{Alt}(\Omega)$; and if $x$ is a generator of $\mathbb{F}_q^{\times,2}$, then one may check that $x$ is a product of two cycles of length $(q-1)/2$, so in particular $x$ is alternating. Finally $F \leq \mathrm{Alt}(\Omega)$.

Now let us take $H = \mathrm{Alt}(\Omega)$ and $H' = \mathrm{Sym}(\Omega)$. The assumptions of Corollary \ref{cor-simpl-ind8} are satisfied for both $G(F,F')$ and $G(H,H')$, and therefore these have open normal simple subgroups $N(F,F')$ and $N(H,H')$ of index eight. Moreover $N(F,F')$ is contained in $N(H,H')$, and according to Corollary \ref{cor-coco-lattice} this inclusion is closed and cocompact. 

Since $q=1 \! \! \mod 4$, the short exact sequence \[ 1 \rightarrow \mathbb{F}_q^{\times,2} \rightarrow \mathbb{F}_q^{\times} \rightarrow C_2 \rightarrow 1 \] does not split. A fortiori \[ 1 \rightarrow \mathrm{GL}^2(n,q) \rightarrow \mathrm{GL}(n,q) \rightarrow C_2 \rightarrow 1 \] does not split either, and it follows that $F_0 = \mathrm{GL}^2(n,q)$ is essential in $F_0' = \mathrm{GL}(n,q)$. Moreover \[ |F_0'| = \prod_{i=0}^{n-1} (q^n-q^i) < q^{n^2} < 2^{q^n-1} =  \left( F_{0}': F_{0} \right) ^{|\Omega|-1},\] so Corollary \ref{cor-without-lattice} applies and shows that $G(F,F')$ (and a fortiori $N(F,F')$) does not have lattices.

To complete the argument, we shall prove that $N(H,H')$, or equivalently $G(H,H')$, contains lattices. To this end, let us assume that $q^n > p$, where $p$ is the characteristic of $\mathbb{F}_q$. Write $\Omega = \left\{1,\ldots,q^n\right\}$, and let $x=q^n / p$. Then $\alpha = \prod_{i=1}^x \left((i-1)p+1,\ldots,ip\right)$ is a product of $p$-cycles with disjoint support, and therefore $K=\left\langle \alpha \right\rangle$ acts freely on $\Omega$, and $K \leq \mathrm{Alt}(\Omega)$. Now consider $\tau = \prod_{i=1}^p \left(i,p+i \right)$. By construction $\tau$ is an involution and $\tau$ commutes with $\alpha$, so $K'=\left\langle \alpha, \tau \right\rangle$ is a extension of $K$ by $C_2$, and $K'$ does not lie in in $\mathrm{Alt}(\Omega)$. Therefore the assumptions of Corollary \ref{cor-coco-lattice} are all fulfilled, so it follows that $G(K,K')$ is a cocompact lattice in $G(H,H')$. 
\end{proof}

We now turn to the proof of Theorem \ref{thm-simple-simple-lattice}. Note that by Proposition \ref{prop-morph-trivial}, the following family of examples actually gives infinitely many pairwise non-isomorphic non-discrete simple groups with simple lattices.

\begin{proof}[Proof of Theorem \ref{thm-simple-simple-lattice}]
Assume that $d \geq 7$ is equal to $0,3 \! \! \mod 4$. Let $H = D_d$ be the dihedral group, $H'= \mathrm{Sym}(d)$, and take $F = D_d \cap \mathrm{Alt}(d)$ and $F' = \mathrm{Alt}(d)$. The assumption on $d$ implies that point stabilizers in $H$ are not contained in $\mathrm{Alt}(d)$, so it follows from Theorem \ref{thm-simple-ind2} that $G(H,H')^\star$ is simple. Moreover $F$ and $F'$ satisfy the assumptions of Corollary \ref{cor-g(f,f')-simple}, so $G(F,F')^\star$ is also simple. Now by definition $H \cap F' = F$ and $H'=HF'$, so by Corollary \ref{cor-coco-lattice} the group $G(F,F')^\star$ is a cocompact lattice in $G(H,H')^\star$. This proves the statement.
\end{proof}

\subsection{Final remark}

The study of the groups $U(F)$ has recently been extended to the case of infinite permutation groups $F$, in which case $U(F)$ acts on a tree that is not locally finite \cite{Smith}. Similarly, given infinite permutation groups $F \leq F'$, we can consider the group $G(F,F')$. While the results from Section \ref{sec-simple} are likely to extend to this framework, some of the results from Section \ref{sec-preliminaries} need some care to be generalized (notably if we wish to obtain a locally compact topology on $G(F,F')$). Nevertheless, investigating further the family of groups $G(F,F')$ in the setting of \cite{Smith} is likely to provide interesting examples.

\subsection*{Added on January 19th, 2016}

After completing this work, we learned that some examples of groups $G(F,F')$ appeared in \cite{Reid-sylow}. In the terminology of \cite{Reid-sylow}, there are examples of $G(F,F')$  which are $p$-localizations of $U(F')$.

\nocite{*}
\bibliographystyle{amsalpha}
\bibliography{almost}

\end{document}